\documentclass[twoside,11pt]{article}
\usepackage{blindtext}

\usepackage{jmlr2e}
\usepackage{algorithm}
\usepackage{algorithmicx}
\usepackage{graphicx}
\usepackage{algpseudocode}
\usepackage{threeparttable}
\usepackage{multirow}
\usepackage{amssymb,amsmath}
\usepackage{paralist}
\usepackage{subfigure}
\usepackage{color}
\newtheorem{assumption}{Assumption}
\allowdisplaybreaks[4]

\ShortHeadings{Zeroth-Order Stochastic Mirror Descent Algorithms for MERO}{Gu and Xu}
\firstpageno{1}
\begin{document}
	
	\title{Zeroth-Order Stochastic Mirror Descent Algorithms for Minimax Excess Risk Optimization\footnotemark[1]}
	
	\author{\name Zhi-Hao Gu \email Wayne02710@163.com \\
		\addr Department of Mathematics\\
		Shanghai University\\
		Shanghai 200444, People's Republic of China
		\AND
		\name Zi Xu\footnotemark[2] \email xuzi@shu.edu.cn \\
		\addr Department of Mathematics\\
		Shanghai University\\
		Shanghai 200444, People's Republic of China\\
		and\\
		Visiting Scholar at Center for Intelligent Computing\\
		Great Bay Institute for Advanced Study\\
		Dongguan, 523000, People's Republic of China}
	\editor{}
	\maketitle
	\renewcommand{\thefootnote}{\fnsymbol{footnote}}
	\footnotetext[1]{This work is supported by National Natural Science Foundation of China under the grants 12071279. }
	\footnotetext[2]{Corresponding author.}
	
	\begin{abstract}
		The minimax excess risk optimization (MERO) problem is a new variation of the traditional distributionally robust optimization (DRO) problem, which achieves uniformly low regret across all test distributions under suitable conditions. In this paper, we propose a zeroth-order stochastic mirror descent (ZO-SMD) algorithm available for both smooth and non-smooth MERO to estimate the minimal risk of each distrbution, and finally solve MERO as (non-)smooth stochastic convex-concave (linear) minimax optimization problems. The proposed algorithm is proved to converge at optimal convergence rates of $\mathcal{O}\left(1/\sqrt{t}\right)$ on the estimate of $R_i^*$ and $\mathcal{O}\left(1/\sqrt{t}\right)$ on the optimization error of both smooth and non-smooth MERO. Numerical results show the efficiency of the proposed algorithm.
	\end{abstract}
	
	\section{Introduction}\label{1}
	With the increasing integration of machine learning into diverse applications, it's frequent to encounter scenarios where the distribution of test data diverges from that of the training data, which leads to marked declines in model efficacy \citep{sugiyama2007covariate}. To counter this issue, Distributionally Robust Optimization (DRO) has been developed as a method to identify models that optimize the worst-case risk across a range of possible distributions. Typically, DRO is articulated as a minimax optimization problem, where the objective is to minimize the maximum potential risk:
	\begin{equation}\label{dro}
		\min_{\mathbf{w} \in \mathcal{W}} \max_{\mathcal{P} \in \mathcal{S}} \mathbb{E}_{\mathbf{z} \sim \mathcal{P}}\left[\ell(\mathbf{w};\mathbf{z})\right],
	\end{equation}
	where $\mathcal{W}$ is a hypothesis class, $\mathcal{S}$ is a set of distributions, and $\ell(\textbf{w}; \textbf{z})$ denotes a loss function that measures the performance, where $\mathbf{w} \in \mathcal{W}$ denotes the $d$-dimentional weight vector and $\mathbf{z} \in \mathcal{P}$ denotes a random sample drawn from a certain distribution $\mathcal{P} \in \mathcal{S}$. When $\mathcal{S}$ contains a finite
	number of distributions, \eqref{dro} is referred to as Group DRO (GDRO) \citep{sagawa2019distributionally}.
	
	While the minimax approach bolsters a model's robustness against shifts in data distribution, it may inadvertently heighten the model's sensitivity to disparate noise. This is particularly pronounced in GDRO, where the considered distributions may display marked variance. The underlying issue is evident: the presence of a single noisy distribution can disproportionately influence the maximization in the formulation, overshadowing the contribution of other, less noisy distributions. To address this concern, Agarwal and Zhang (\citet{agarwal2022minimax}) introduced an alternative called Minimax Regret Optimization (MRO). This technique refines the standard risk metric, by focusing on the excess risk:
	\begin{equation}
		\mathbb{E}_{\mathbf{z} \sim \mathcal{P}}\left[\ell(\mathbf{w}; \mathbf{z})\right]-\min_{\mathbf{w} \in \mathcal{W}} \mathbb{E}_{\mathbf{z} \sim \mathcal{P}}\left[\ell(\mathbf{w}; \mathbf{z})\right]. \nonumber 
	\end{equation}
	
	Specifically, they examine the context of GDRO and aim to reduce the highest possible excess risk across a collection of $m$ distributions, denoted as $\mathcal{P}_{1},\cdots,\mathcal{P}_{m}$:
	\begin{equation}\label{mero1}
		\min_{\mathbf{w} \in \mathcal{W}} \max_{i \in [m]} \Big\{\underbrace{\mathbb{E}_{\mathbf{z} \sim \mathcal{P}_{i}} \left[ \ell(\mathbf{w}; \mathbf{z}) \right]}_{:=R_{i}(\mathbf{w})}  - \underbrace{\min_{\mathbf{w} \in \mathcal{W}} \mathbb{E}_{\mathbf{z} \sim \mathcal{P}_{i}} \left[ \ell(\mathbf{w}; \mathbf{z}) \right]}_{:=R_{i}^{*}} \Big\},
	\end{equation}
	where we denote $R_{i}(\mathbf{w})=\mathbb{E}_{\mathbf{z} \sim \mathcal{P}_{i}}\left[ \ell(\mathbf{w}; \mathbf{z}) \right]$ and $R_{i}^{*}=\min_{\mathbf{w} \in \mathcal{W}} \mathbb{E}_{\mathbf{z} \sim \mathcal{P}_{i}} \left[ \ell(\mathbf{w}; \mathbf{z}) \right]$ for simplicity, and $i \in [m]$ denotes the performance of locating the highest excess risk under $\mathcal{P}_i$ among $m$ distributions.
	Under this setting, Zhang and Tu (\citet{zhang2023efficient}) evolved MRO into minimax excess risk optimization (MERO) to avoid confusion with the term “regret” commonly used in online learning \citep{cesa2006prediction}.
	
	MERO can be described as deducting each distribution's inherent level of difficulty i.e., the minimum risk $R_i^*$ from its corresponding risk, thereby normalizing the risks for better comparability. However, the actual value of $R_i^*$
	often remains elusive, which introduces considerable complexity to the optimization task outlined in \eqref{mero1}. In the previous work, Agarwal and Zhang (\citet{agarwal2022minimax}) tackled this by focusing on the empirical variant of MERO. Yet, the derived algorithm necessitates solving an empirical risk minimization problem during each iteration---a task that becomes exceedingly demanding under a large dataset. Pivoting from this groundwork, Zhang and Tu (\citet{zhang2023efficient}) proposed novel and practical stochastic approximation methods that directly engage with optimization \eqref{mero1}. Notably, this approach also obviates the necessity to reconcile empirical MERO discrepancies with its population counterpart. By leveraging stochastic approximation techniques \citep{nemirovski2009robust}, they reformulated \eqref{mero1} into a stochastic convex-concave saddle-point challenge:
	\begin{equation}\label{mero2}
		\min_{\mathbf{w} \in \mathcal{W}} \max_{\mathbf{q} \in \Delta_{m}} \phi(\mathbf{w},\mathbf{q}):=\sum_{i=1}^{m} q_i\left[R_{i}(\mathbf{w}) - R_{i}^{*}\right],
	\end{equation}
	where $\Delta_{m}=\left\{\mathbf{q} \in \mathbb{R}_{m}:\mathbf{q}\geq0, \sum_{i=1}^{m}q_i=1\right\}$ is the $(m-1)$-dimensional simplex.
	\subsection{Related Works}\label{1.1}
	\noindent For the DRO problem mentioned in \eqref{dro}, the uncertainty set $\mathcal{S}$ is typically chosen to be a set of distributions neighboring a target distribution $\mathcal{P}_{0}$, and constructed by using a specific distance function between distributions, such as the $f$-divergence \citep{ben2013robust} or Wasserstein distance \citep{mohajerin2018data,kuhn2019wasserstein}. The nature of the loss function also exhibits variability: it can be assumed to have a convex structure \citep{ben2015oracle,shapiro2017distributionally}, a nonconvex structure \citep{jin2021non,qi2021online}, or may include regularization terms \citep{sinha2017certifying}, among others. Research efforts also focus on a variety of goals, such as developing optimization algorithms \citep{namkoong2016stochastic,levy2020large,rafique2022weakly}, investigating the finite sample sizes and asymptotic behaviors of empirical solutions \citep{duchi2019variance,duchi2021learning}, establishing confidence intervals for risk assessments \citep{duchi2021statistics}, or approximating nonparametric likelihoods \citep{nguyen2019optimistic}.
	
	If the distribution set $\mathcal{S}$ is finite, the traditional DRO problem evolves into GDRO, which can be written in the form of the following stochastic convex-concave problem:
	\begin{equation}\label{gdro}
		\min_{\mathbf{w} \in \mathcal{W}} \max_{\mathbf{q} \in \Delta_{m}} \phi(\mathbf{w},\mathbf{q}) = \sum_{i=1}^{m} q_i R_{i}(\mathbf{w}).
	\end{equation}
	By adopting the method of taking a random sample from each sample in every round, Sagawa et al. (\citet{sagawa2019distributionally}) applied the stochastic mirror descent (SMD) algorithm to \eqref{gdro} and achieved a convergence rate of $\mathcal{O}\left(m\sqrt{(\log m)/t}\right)$, where $t$ is the number of iteration, and ultimately obtaining an $\epsilon$-optimal solution with a sample complexity of $\mathcal{O}\left(m\sqrt{(\log m)/\epsilon^2}\right)$.
	Recently, Zhang and Tu (\citet{zhang2023efficient}) proposed the MERO problem as formulated in \eqref{mero2}, which is also the focus of this paper. They present the multi-stage stochastic approximation method and anytime stochastic approximation method to solve MERO and achieved an iteration complexity of $\mathcal{O}\left((\log^2 t + \log^{1/2}m \log^{3/2}t)/\sqrt{t}\right)$, where $t$ is the number of iteration.
	
	The mirror descent method \citep{censor1992proximal,beck2003mirror} serves as an effective optimization technique within the realm of machine learning due to its capability to adapt to the geometric structure of optimization challenges through the judicious selection of appropriate Bregman functions \citep{bregman1967relaxation,censor1981iterative}. Recently, Li et al. (\citet{li2020variance}) proposed a variance-reduced adaptive stochastic mirror descent algorithm to tackle nonsmooth and nonconvex finite-sum optimization problems. Moreover, the application of the mirror descent method has further been extended to address minimax optimization challenges. For example, Babanezhad and Lacoste-Julien (\citet{babanezhad2020geometry}) proposed an algorithm rooted in mirror descent principles for tackling convex-concave minimax scenarios. Rafique et al. (\citet{rafique2022weakly}) designed a series of mirror descent strategies suitable for minimax optimization in weakly convex contexts. Additionally, a novel approach based on mirror descent principles \citep{mertikopoulos2018optimistic} has emerged for addressing certain nonconvex-nonconcave minimax issues characterized by non-monotonic variational inequalities. More recently, Paul et al. (\citet{paul2023robust}) proposed almost sure convergence of zeroth-order mirror descent algorithm and can find an $\epsilon$-stationary point with the total complexity of $\mathcal{O}(\epsilon^{-2})$. Zhang and Tu (\citet{zhang2023efficient}) proposed a stochastic mirror descent algorithm to solve a class of DRO problems (MERO). 
	
	In addition to first-order algorithms to solve DRO problems, there are some existing works that focus on zeroth-order algorithms for solving minimax optimization problems. For example, under convex-concave setting, Maheshwari et al. (\citet{maheshwari2022zeroth}) proposed a gradient descent-ascent algorithm with random reshuffling, i.e., OGDA-RR and obtained an $\epsilon$-stationary point with the total complexity of $\mathcal{O}(\epsilon^{-4})$. Sadiev et al. (\citet{sadiev2021zeroth}) proposed ZO-ESVIA algorithm and ZO-SCESVIA algorithm, which can find an $\epsilon$-stationary point with iteration complexity of $\mathcal{O}(\epsilon^{-2})$ and $\tilde{\mathcal{O}}(\epsilon^{-1})$ respectively. Under nonconvex-concave setting, Xu et al. (\citet{xu2021derivative}) proposed the ZO-AGP algorithm and its iteration complexity of obtaining an $\epsilon$-staitionary point is bounded by $\mathcal{O}(\epsilon^{-4})$. Under nonconvex-strongly concave setting, Wang et al. (\citet{wang2023zeroth}) proposed two single-loop algorithms, i.e., ZO-GDA and ZO-GDMSA (resp. ZO-SGDA and ZO-SGDMSA), and the total number of calls
	of the zeroth-order oracle to obtain an $\epsilon$-stationary point is bounded by $\mathcal{O}\left(\kappa^5(d_1+d_2)\epsilon^{-2}\right)$(resp. $\mathcal{O}\left(\kappa^5(d_1+d_2)\epsilon^{-4}\right)$ and $\mathcal{O}\left(\kappa(d_1+\kappa d_2 \log (\epsilon^{-1}))\epsilon^{-2}\right)$(resp. $\mathcal{O}\left(\kappa(d_1+\kappa d_2 \log (\epsilon^{-1}))\epsilon^{-4}\right)$ respectively, where $\kappa$ is the condition number. Liu et al. (\citet{liu2019min,liu2020min}) proposed an alternating projected stochastic gradient descent-ascent method called ZO-Min-Max, which can find an $\epsilon$-stationary point with the total complexity of $\mathcal{O}((d_1+d_2)\epsilon^{-6})$. Under nonconvex-linear setting, Shen et al. (\citet{shen2023zeroth}) proposed a zeroth-order single-loop algorithm to find an $\epsilon$-first-order Nash equilibrium bounded by $\mathcal{O}(\epsilon^{-3})$ with the number
	of function value estimation per iteration being bounded by $\mathcal{O}(d_x\epsilon^{-2})$. 
	
	Finally, we give a brief introduction about non-smooth optimization related to zeroth-order gradient. Under stochastic convex setting, Nesterov and Spokoiny (\citet{nesterov2017random}) and Duchi et al. (\citet{duchi2015optimal}) proposed $\mathcal{R}\mathcal{S}_{\mu}$ algorithm and two-point gradient estimates algorithm based on mirror descent method which achieve an $\epsilon$-stationary point with the complexity of $\mathcal{O}\left(d^2\epsilon^{-2}\right)$ and $\mathcal{O}\left(d\log(d)\epsilon^{-2}\right)$ respectively. Rando et al. (\cite{rando2022stochastic}) proposed S-SZD (Structured Stochastic Zeroth order Descent) algorithm and attain a convergence rate on the function values of the form $\mathcal{O}\left(d/lt^{-c}\right)$ for every $c<1/2$, where $t$ is the number of iterations. Under stochastic non-convex setting, both Rando et al. (\citet{rando2022stochastic}) and Kozak et al. (\citet{kozak2023zeroth}) focused on objectives satisfying the Polyak-Łojasiewicz (PL) conditions. The former proved that S-SZD algorithm could obtain linear convergence and the latter proposed a zeroth order stochastic subspace algorithm and finally reached polynomial convergence and linear convergence by different choice of discretization parameter. Lin et al. (\citet{lin2022gradient}) proposed GFM and SGFM algorithms and achieve the same $(\delta,\epsilon)$-Goldstein stationary point at an expected convergence rate at $\mathcal{O}\left(d^{3/2}\delta^{-1}\epsilon^{-4}\right)$.
	
	\section{Preliminaries}\label{2}
	\subsection{Notations}\label{2.1}
	\noindent Throughout the paper, we use the following notations. $\langle x,y \rangle$ denotes the inner product of two vectors of $x$ and $y$. $\|\cdot\|$ denotes the norm of a vector. We use $\mathbb{R}^{d}$ to denote the space of $d$ dimension real valued vectors. $\mathbb{E}_u[\cdot]$ means the expectation over the random vector $u$, $\mathbb{E}_{(u,\xi)}[\cdot]$ means the joint expectation over the random vector $u$ and the random
	variable $\xi$, and $\mathbb{E}_{(U,\mathcal{B})}$ denotes the joint expectation over the set $U$ of random vectors and
	the set $\mathcal{B}$ of random variables $\{\xi_1,\cdots,\xi_b\}$.
	\subsection{Zeroth-Order Gradient Estimator}\label{2.2}
	\noindent We consider the general stochastic convex optimization problem: 
	\begin{equation}\label{min}
		\min_xf(x) := \mathbb{E}_{\xi \sim \mathcal{P}} \left[F(x ; \xi)\right]. \nonumber
	\end{equation}
	
	Suppose that the gradient information is not available directly,
	we first introduce the idea of uniform smoothing gradient estimator (UniGE) \citep{gao2018information} for smooth case, i.e., $F(x ; \xi)$ has Lipschitz continuous gradient. Specifically, given \(\mathcal{B}=\left\{\xi_1,\cdots,\xi_r\right\}\) drawn i.i.d. from an unknown distribution \(\mathcal{P}\), the UniGE of \(\nabla F(x ; \xi_i)\) is defined as
	\begin{equation}\label{smooth unige}
		\hat{\nabla} F(x ; \mathcal{B}) = \frac{1}{r} \sum_{i=1}^{r} \hat{\nabla} F(x ; \xi_i) = \frac{1}{r} \sum_{i=1}^{r} \frac{F(x + \mu u_i; \xi_i) - F(x; \xi_i)}{\mu / d} u_i ,
	\end{equation}
	where $u_i \in \mathbb{R}^{d}$ are random vectors generated from the uniform distribution over $d$-dimensional unit sphere, $\mu>0$ is a smoothing parameter. We also denote that
	\begin{equation}\label{smooth f}
		f_{\mu}(x) = \mathbb{E}_{(u,\xi)}\left[ F(x + \mu u; \xi)\right].
	\end{equation}
	Note that for any random variable $\xi$, we have $\mathbb{E}_{(u,\xi)}\left[\hat{\nabla}F(x; \xi)\right] = \nabla_x f_{\mu}(x)$ by Lemma 5 in \citep{ji2019improved}. Then, obviously we have $\mathbb{E}_{(U,\mathcal{\mathcal{B}})}\left[\hat{\nabla}F(x; \mathcal{B})\right] = \nabla f_{\mu}(x)$ where $U =
	\{u_1,\cdots,u_r\}$.
	
	Additionally, for non-smooth case, we introduce the non-smooth directional gradient estimator \citep{duchi2015optimal} as
	\begin{equation}\label{non-smooth unige}
		\hat{\nabla}_{\text{ns}} F(x; \mathcal{B}) = \frac{1}{r} \sum_{i=1}^{r} \hat{\nabla}_{\text{ns}} F(x; \xi_i) = \frac{1}{r} \sum_{i=1}^{r} \frac{F(x + \mu_1 u_i + \mu_2 v_i; \xi_i) - F(x + \mu_1 u_i; \xi_i)}{\mu_2} v_i ,
	\end{equation}
	where $u_i, v_i \in \mathbb{R}^{d}$ are random vectors generated from smoothing distributions which are one of the following pairs: (1) both distributions are standard normal in $\mathbb{R}^{d}$ with identity covariance, (2) both distributions are uniform on the $\ell_2$-ball of radius $\sqrt{d+2}$, (3) one distribution is uniform on the $\ell_2$-ball of radius $\sqrt{d+2}$, and the other distribution is uniform on the $\ell_2$-sphere of radius $\sqrt{d}$. $\mu_1, \mu_2 > 0$ are smoothing parameters. Similarly as we construct the Unige in the smooth case \eqref{smooth unige}, we introduce an additional step of randomization for the point at which we evaluate the function difference. Essentially, this randomness reduces the likelihood that the perturbation vector is near a point of non-smoothness, allowing us to apply results similar to those in the smooth case.
	\section{Zeroth-Order Stochastic Mirror Descent Algorithm}\label{3}
	\noindent The stochastic mirror descent (SMD) algorithm is a well-known method for solving minimax problems, especially for the distributionally robust optimization (DRO) problems. The exsiting SMD algorithm needs to know the gradient information of the objective function, however there are many cases where the gradient value of the objective function is hard to compute or entirely unknown. In this section, we propose a zeroth-order stochastic mirror descent (ZO-SMD) algorithm to tackle MERO problems under both smooth and non-smooth cases. 
	
	We first present the standard setup of the mirror descent \citep{nemirovski2009robust}. We endow the structural domain $ \mathcal{W} $ with a distance-generating function that is 1-strongly convex w.r.t. a specific norm $ \|\cdot\|_{w}$. Then, we define the Bregman distance associated with $\nu_{w}(\cdot)$ as:
	\begin{equation}\label{bregman}
		B_{w}(\mathbf{u},\mathbf{v})=\nu_{w}(\mathbf{u})-\left[\nu_{w}(\mathbf{v})+\left\langle \nabla \nu_{w}(\mathbf{v}), \mathbf{u}-\mathbf{v} \right\rangle\right]. \nonumber
	\end{equation}
	For the simplex \( \Delta_{m} \), we choose the entropy function \( \nu_{q}(\mathbf{q})=\sum_{i=1}^{m}q_{i}\ln q_{i} \), which is a distance-generating function that is 1-strongly convex w.r.t. the $l_{1}$ vector norm $ \|\cdot\|_{1}$. 
		It is easy to deduce that the Bregman distance is bounded w.r.t. $\mathbf{w}$, i.e.,
	\begin{equation}\label{B_w bound}
		\max_{\mathbf{w}\in \mathcal{W}}B_{w}(\mathbf{w},\mathbf{o}_{w})\leq D^{2}, 
	\end{equation}
	where $\mathbf{o}_{w}=\operatorname{argmin}_{\mathbf{w} \in \mathcal{W}}\nu_{w}(\mathbf{w})$.
	\noindent For the simplex $\Delta_{m}$, we have 
	\begin{equation}\label{B_q bound}
		\max_{\mathbf{q}\in\Delta_{m}}B_{q}(\mathbf{q},\mathbf{o}_{q})\leq \ln m,
	\end{equation} 
	where $\textbf{o}_{q}=\frac{1}{m}\mathbf{1}_{m}$ and $\mathbf{1}_{m}$ is the
	$m$-dimensional vector of all ones by Proposition 5.1 in \citep{beck2003mirror}.
	The benefit of this selection is underscored by \citep{nemirovski2009robust}, who applied this method to successfully resolve convex optimization problems on a simplex of several million dimensions, proving that such mirror descent algorithms are nearly "free" in terms of the approximation steps on the simplex when the entropy function is chosen as the distance-generating function. In a similar vein, the Bregman distance corresponding to $ v_{q}(\mathbf{q})$, denoted as $B_{q}(\cdot,\cdot)$, also known as the KL divergence, is used to measure the disparity between different probability distributions.
	\subsection{ZO-SMD for Smooth MERO}\label{3.1}
	In this section, we propose a ZO-SMD algorithm to solve smooth MERO problems and analyze its iteration complexity. 
	
	The main idea of the proposed ZO-SMD algorithm is as follows.
	Firstly, we estimate the value of $R_i^*$ in \eqref{mero2} by using the zeroth-order stochastic mirror descent (ZO-SMD) algorithm. Detailedly, in order to update $\mathbf{w}$ under the distribution $\mathcal{P}_i$ in the $t$-th round, i.e., $\mathbf{w}_t^{(i)}$, we replace the stochastic gradient $\nabla \ell(\mathbf{w}_t^{(i)}; \mathbf{z}_t^{(i)})$ in the SMD algorithm \citep{zhang2023efficient} with the zeroth-order gradient $\hat{\nabla} \ell(\mathbf{w}_t^{(i)}; \mathcal{P}_{i})$ to address scenarios where the loss function gradient is difficult to compute or entirely unknown. According to \eqref{smooth unige}, the zeroth-order gradient of $\ell(\mathbf{w};\mathbf{z})$ under the distribution $\mathcal{P}_{i}$ for all $i \in [m]$ w.r.t. $\mathbf{w}$ in the $t$-th round is defined as:
	\begin{equation}\label{smooth unige of loss function}
		\hat{\nabla} \ell(\mathbf{w}_t^{(i)}; \mathcal{P}_{i}) = \frac{1}{r} \sum_{j=1}^{r} \hat{\nabla} \ell(\mathbf{w}_t^{(i)}; \mathbf{z}_{t,i}^{(j)})= \frac{1}{r} \sum_{j=1}^{r} \frac{\ell(\mathbf{w}_t^{(i)} + \mu_t \mathbf{u}_j; \mathbf{z}_{t,i}^{(j)}) - \ell(\mathbf{w}_t^{(i)}; \mathbf{z}_{t,i}^{(j)})}{\mu_t/d} \mathbf{u}_j,
	\end{equation}
	where $\mathbf{u}_j \in \mathbb{R}^d$ is a random vector drawn from the uniform distribution on the $d$-dimensional unit sphere $B_w$, which is defined as $B_w = \{ \mathbf{u} \in \mathbb{R}^d : \|\mathbf{u}\|_2 \leq 1 \}$, and $\{\mu_t\}_{t=1}^{\infty}$ is a non-increasing sequence of positive
	smoothing parameters, $\mathbf{z}_{t,i}^{(1)}, \cdots, \mathbf{z}_{t,i}^{(r)}$ are $r$ random samples selected from $\mathcal{P}_i$ in the $t$-th round. Furthermore, due to the equivalence of norms in finite-dimensional vector spaces, without loss of generality we assume $\|\cdot\|_{w}\leq \tau_1\|\cdot\|_2$ and $\|\cdot\|_{w,*}\leq \tau_2\|\cdot\|_2$.
	
	The ZO-SMD algorithm using the following update to estimate $R_i^*$ for all $i \in [m]$ step by step:
	\begin{equation}\label{ZO-SMD for smooth min}
		\mathbf{w}_{t+1}^{(i)} = \underset{\mathbf{w} \in \mathcal{W}}{\operatorname{argmin}} \left\{ \eta_t^{(i)} \left\langle \hat{\nabla} \ell(\mathbf{w}_t^{(i)}; \mathcal{P}_{i}), \mathbf{w} - \mathbf{w}_t^{(i)} \right\rangle + B_{w}\left(\mathbf{w}, \mathbf{w}_t^{(i)}\right) \right\},
	\end{equation}
	where $\eta_t^{(i)} > 0$ is the step size. We use the partial averaging value $R_i(\bar{\mathbf{w}}_t^{(i)})$ as 
	as an approximate solution for $R_i^*=\min_{\mathbf{w} \in \mathcal{W}} R_i(\mathbf{w})$, where $\bar{\mathbf{w}}_t^{(i)}$ is defined as:
	\begin{equation}\label{weighted value for smooth min}
		\bar{\mathbf{w}}_t^{(i)} = \sum_{j=\lceil\frac{t}{2}\rceil}^{t} \frac{\eta_j^{(i)} \mathbf{w}_j^{(i)}}{\sum_{k=\lceil\frac{t}{2}\rceil}^{t} \eta_k^{(i)}},
	\end{equation}
	where $\lceil\cdot\rceil$ is the ceiling function. The reason why we do not choose the last iteration $\textbf{w}_t^{(i)}$ as the approximate solution is that this will lead to a more complicated analysis \citep{shamir2013stochastic,harvey2019tight,jain2019making}. For simplicity, we denote $\hat{\phi}(\textbf{w},\textbf{q})=\sum_{i=1}^{m}q_i\left[R_i(\textbf{w})-R_i(\bar{\textbf{w}}_t^{(i)})\right]$.
	
	We proceed to solve \eqref{mero2} by ZO-SMD. Let $\mathbf{w}_t$ and $\mathbf{q}_t$ be the solutions in the $t$-th round. Based on the mini-batch random samples $\mathbf{z}_{t,i}^{(1)}, \cdots, \mathbf{z}_{t,i}^{(r)}$ generated from $\mathcal{P}_i$ in the estimate of $R_i^*$, we define the zeroth-order gradient of $\phi(\mathbf{w},\mathbf{q})$ w.r.t. $\mathbf{w}$ in the $t$-th round as
	\begin{equation}\label{gradient_w for smooth minimax}
		\hat{\mathbf{g}}_{w}(\mathbf{w}_t, \mathbf{q}_t) = \sum_{i=1}^{m} q_{t,i} \hat{\nabla}\ell(\mathbf{w}_{t}; \mathcal{P}_i) = \sum_{i=1}^{m} q_{t,i}\frac{1}{r} \sum_{j=1}^{r} \frac{\ell(\mathbf{w}_t + \mu_t \mathbf{u}_j; \mathbf{z}_{t,i}^{(j)}) - \ell(\mathbf{w}_t; \mathbf{z}_{t,i}^{(j)})}{\mu_t/d} \mathbf{u}_j.
	\end{equation}
	Considering that $\phi(\mathbf{w},\mathbf{q})$ is linear w.r.t. $\mathbf{q}$, there is no need to construct a zeroth-order gradient w.r.t. $\mathbf{q}$. Thus, based on the mini-batch random samples $\mathbf{z}_{t,i}^{(1)}, \cdots, \mathbf{z}_{t,i}^{(r)}$ generated in the estimate of $R_i^*$, we maintain the stochastic gradient of $\hat{\phi}(\mathbf{w},\mathbf{q})$ w.r.t. $\mathbf{q}$ in the $t$-th round as:
	\begin{equation}\label{gradient_q for smooth minimax}
		\mathbf{g}_q(\mathbf{w}_t, \mathbf{q}_t) = \left[\frac{1}{r} \sum_{j=1}^{r}\left[\ell(\mathbf{w}_t; \mathbf{z}_{t,1}^{(j)}) - \ell(\bar{\mathbf{w}}_t^{(1)},\mathbf{z}_{t,1}^{(j)})\right], \cdots, \frac{1}{r} \sum_{j=1}^{r}\left[\ell(\mathbf{w}_t; \mathbf{z}_{t,m}^{(j)}) - \ell(\bar{\mathbf{w}}_t^{(m)},\mathbf{z}_{t,m}^{(j)})\right] \right]^\top.
	\end{equation}
	Obviously, $\mathbf{g}_q(\mathbf{w}_t, \mathbf{q}_t)$ is a biased estimator of the gradient of $\phi(\mathbf{w}_t,\mathbf{q}_t)$, i.e., $\nabla_q\phi(\mathbf{w}_t,\mathbf{q}_t)=\left[ R_1(\mathbf{w}_t) - R_1^*, \cdots, R_m(\mathbf{w}_t) - R_m^* \right]^\top$, and the bias determined by $R_i(\bar{\mathbf{w}}_t^{(1)})-R_i^*$ can be controlled according to Lemma 4 in \citep{zhang2023efficient}. Note that $\mathbf{g}_q(\mathbf{w}_t, \mathbf{q}_t)$ is a unbiased estimator of  $\nabla_q\hat{\phi}(\mathbf{w}_t,\mathbf{q}_t)$.
	
	Finally, equipped with the zeroth-order gradient and the stochastic gradient in \eqref{gradient_w for smooth minimax} and \eqref{gradient_q for smooth minimax}, we update $\mathbf{w}_t$ and $\mathbf{q}_t$ by
	\begin{equation}\label{ZO-SMD w for smooth minimax}
		\mathbf{w}_{t+1}=\underset{\mathbf{w} \in \mathcal{W}}{\operatorname{argmin}}\left\{\eta_{t}^{w}\left\langle \hat{\mathbf{g}}_{w}(\mathbf{w}_t,\mathbf{q}_t),\mathbf{w}-\mathbf{w}_{t}\right\rangle + B_w\left(\mathbf{w}, \mathbf{w}_{t}\right)\right\},
	\end{equation}
	\begin{equation}\label{ZO-SMD q for smooth minimax}
		\mathbf{q}_{t+1}=\underset{\mathbf{q} \in \Delta{m}}{\operatorname{argmin}}\left\{\eta_{t}^{q}\langle -\mathbf{g}_{q}(\mathbf{w}_t,\mathbf{q}_t),\mathbf{q}-\mathbf{q}_{t}\rangle + B_q(\mathbf{q}, \mathbf{q}_{t})\right\},
	\end{equation}
	where $\eta_{t}^{w} > 0$ and $\eta_{t}^{q} > 0$ are step sizes. The partial averaging value of iterates is also used, i.e.,
	\begin{equation}\label{weighted value for smooth minimax}
		\bar{\mathbf{w}}_t = \sum_{j=\lceil\frac{t}{2}\rceil}^{t}\frac{\eta_j^{w} \mathbf{w}_j}{\sum_{k=\lceil\frac{t}{2}\rceil}^{t} \eta_k^{w}},\quad \bar{\mathbf{q}}_t = \sum_{j=\lceil\frac{t}{2}\rceil}^{t}\frac{\eta_j^{q} \mathbf{q}_j}{\sum_{k=\lceil\frac{t}{2}\rceil}^{t} \eta_k^{q}}.
	\end{equation}
	The detailed algorithm is presented as in Algorithm \ref{alg:1}.  In following subsections, we analyze the iteration complexity of the zeroth-order stochasitc mirror descent (ZO-SMD) algorithm for solving smooth MERO. 
	\begin{algorithm}[H]
		\caption{Zeroth-Order Stochastic Mirror Descent Algorithm: Smooth Case}\label{alg:1}
		\begin{algorithmic}[1]
			\State Initialize $\mathbf{w}_1 = \textbf{w}_1^{(1)} =\cdots= \mathbf{w}_1^{(m)} = {\operatorname{argmin}}_{\mathbf{w}\in \mathcal{W}} v_w(\mathbf{w})$, and $\mathbf{q}_1 = \frac{1}{m}\mathbf{1}_m \in \mathbb{R}^m$.
			\For{$t = 1$ to $T$}
			\For{$\forall$ $i \in [m]$}
			\State  Draw $r$ random samples $ \mathbf{z}_{t,i}^{(1)},\cdots, \mathbf{z}_{t,i}^{(r)}$ from $\mathcal{P}_i$.
			\EndFor
			\For{$\forall$ $i \in [m]$}
			\State Calculate $\hat{\nabla} \ell (\mathbf{w}_t^{(i)}; \mathcal{P}_i)$ according to \eqref{smooth unige of loss function} and update $\mathbf{w}_t^{(i)}$ according to \eqref{ZO-SMD for smooth min}.
			\EndFor
			\For{$\forall$ $i \in [m]$}
			\State Calculate the weighted average $\bar{\mathbf{w}}_t^{(i)}$ in \eqref{weighted value for smooth min}.
			\EndFor
			\State Calculate $\hat{\mathbf{g}}_w(\mathbf{w}_t,\mathbf{q}_t)$ and $\mathbf{g}_q(\mathbf{w}_t,\mathbf{q}_t)$ according to \eqref{gradient_w for smooth minimax} and \eqref{gradient_q for smooth minimax}.
			\State Update $\mathbf{w}_t$ and $\mathbf{q}_t$ according to \eqref{ZO-SMD w for smooth minimax} and \eqref{ZO-SMD q for smooth minimax}.
			\State Calculate the weighted average $\bar{\mathbf{w}}_t$ and $\bar{\mathbf{q}}_t$ in \eqref{weighted value for smooth minimax}.
			\EndFor
		\end{algorithmic}
	\end{algorithm}
	\subsubsection{Technical Preparation}\label{3.1.1}
	\noindent  Firstly, we make some mild assumptions.
	\begin{assumption}\label{assum:1}
		All the risk functions $R_{1}(\mathbf{w}),\cdots,R_{m}(\mathbf{w})$ are convex and the domain $\mathcal{W}$ is convex and compact.
	\end{assumption}
	\begin{assumption}\label{assum:2}
		\textit{For all $i \in [m]$}, for any $\mathbf{w} \in \mathcal{W}$ and $\mathbf{z} \sim \mathcal{P}_{i}$, we have
		\begin{equation}\label{loss function bound}
			0 \leq \ell(\mathbf{w}; \mathbf{z}) \leq C. 
		\end{equation}
	\end{assumption}
	\begin{assumption}\label{assum:3}
		For all $i \in [m]$, $\ell(\mathbf{w};\mathbf{z})$ has Lipschitz continuous gradient, i.e., there exists a constant $L>0$ such that for any $\mathbf{w}_1, \mathbf{w}_2\in \mathcal{W}$ and $\mathbf{z} \sim \mathcal{P}_i$, we have
		\begin{equation}\label{L continuous gradient}
			\|\nabla \ell(\mathbf{w}_1;\mathbf{z})-\nabla \ell(\mathbf{w}_2;\mathbf{z})\|_{w,*} \leq L\|\mathbf{w}_1-\mathbf{w}_2\|_{w}, 
		\end{equation}
		and
		\begin{equation}\label{expected gradient bound}
			\|\nabla R_i(\mathbf{w})\|_{w,*}\leq G,
		\end{equation}
		where $\|\cdot\|_{w,*}$ is the dual norm of $\|\cdot\|_{w}$.
	\end{assumption}
	Note that by Lemma 7 in \citep{xu2023zeroth}, $R_i(\mathbf{w})$ also has Lipschitz continuous gradient with constant $L$.
	\begin{assumption}\label{assum:4}
		For all $i \in [m]$, the variance of the stochastic gradient is bounded, i.e., there exists a constant $\sigma>0$ such that for any $\mathbf{w}\in \mathcal{W}$ and $\mathbf{z} \sim \mathcal{P}_{i}$, we have
		\begin{equation}\label{variance bound}
			\mathbb{E}_{\mathbf{z}\sim \mathcal{P}_i}\left[\left\| \nabla \ell(\mathbf{w}; \mathbf{z}) - \nabla R_i(\mathbf{w}) \right\|_{w,*}^2\right]  \leq \sigma^2.
		\end{equation}
	\end{assumption}
	Given a solution $(\bar{\mathbf{w}},\bar{\mathbf{q}})$ to \eqref{mero2}, the optimization error is defined as:
	\begin{equation}
		\epsilon_{\phi}(\bar{\mathbf{w}},\bar{\mathbf{q}})=\max_{\mathbf{q}\in\Delta_{m}}\phi(\bar{\mathbf{w}},\mathbf{q})-\min_{\mathbf{w} \in\mathcal{W}}\phi(\mathbf{w},\bar{\mathbf{q}}), \nonumber
	\end{equation}
	which provides an upper bound for the excess risk on any distribution, since
	\begin{align}
		&\max_{i \in [m]}[R_{i}(\bar{\mathbf{w}}) - R_{i}^{*}] - \min_{\mathbf{w} \in \mathcal{W}} \max_{\mathbf{q} \in \Delta_{m}} \sum_{i=1}^{m} q_i[R_{i}(\mathbf{w}) - R_{i}^{*}] \nonumber \\
		\leq & \max_{i \in [m]}[R_{i}(\bar{\mathbf{w}}) - R_{i}^{*}] - \min_{\mathbf{w} \in \mathcal{W}}\sum_{i=1}^{m}\bar{q_i}[R_{i}(\mathbf{w}) -R_{i}^{*}]  = \epsilon_{\phi}(\bar{\mathbf{w}},\bar{\mathbf{q}}).\nonumber
	\end{align}
	
	\begin{lemma}\label{lem:1}
		If Assumptions \ref{assum:3} and \ref{assum:4} hold, for all $i \in [m]$ and any $\mathbf{w} \in \mathcal{W}$, we have
		\begin{equation}\label{lem1.1}
			\mathbb{E}_{(\mathbf{u},\mathbf{z})}\left[\hat{\nabla} \ell(\mathbf{w}; \mathbf{z})\right] = \nabla R_i\left(\mathbf{w}\right) + \frac{\tau_1^2\mu dL}{2}v(\mu,\mathbf{w}),
		\end{equation}
		where $v(\mu,\mathbf{w}) \in \mathbb{R}^d$ is an error vector with $\left\|v(\mu,\mathbf{w})\right\|_{w,*} \leq \tau_2$, and
		\begin{equation}\label{lem1.2}
			\mathbb{E}_{(\mathbf{u},\mathbf{z})} \left[\left\| \hat{\nabla} \ell(\mathbf{w}; \mathbf{z}) \right\|_{w,*}^2\right] \leq 4\tau_1^2\tau_2^2(G^2+\sigma^2)d^2 + \frac{\tau_1^4\tau_2^2\mu^2 d^2 L^2}{2}.
		\end{equation}
	\end{lemma}
	\begin{proof}
		Similar to the proof of Lemma 1 in \citep{duchi2015optimal}, by \eqref{L continuous gradient} in Assumption \ref{assum:3}, we obtain
		\begin{equation}
			\frac{\ell(\mathbf{w}+\mu \mathbf{u}; \mathbf{z}) - \ell(\mathbf{w};\mathbf{z})}{\mu/d} \leq \frac{\langle \nabla \ell(\mathbf{w}; \mathbf{z}), \mu \mathbf{u} \rangle + \frac{L}{2} \|\mu \mathbf{u}\|_w^2}{\mu/d} = d\langle \nabla \ell(\mathbf{w}; \mathbf{z}),  \mathbf{u} \rangle + \frac{\tau_1^2\mu dL}{2}. \nonumber
		\end{equation}
		Thus, we have
		\begin{equation}
			\mathbb{E}_{\mathbf{u} \sim U_{B_w}}\left[\frac{\ell(\mathbf{w}+\mu \mathbf{u}; \mathbf{z}) - \ell(\mathbf{w};\mathbf{z})}{\mu/d}\mathbf{u}\right] = \mathbb{E}_{\mathbf{u} \sim U_{B_w}}\left[d\langle \nabla \ell(\mathbf{w}; \mathbf{z}),  \mathbf{u} \rangle\mathbf{u} + \frac{\tau_1^2\mu dL}{2} \gamma(\mu,\mathbf{w},\mathbf{u})\mathbf{u}\right], \nonumber
		\end{equation}
		where $\gamma(\mu,\mathbf{w},\mathbf{u})$ is some function with range contained in $[0,1]$. Since $\mathbf{u} \in \mathbb{R}^d$ is a random vector drawn from the uniform distribution on the $d$-dimensional unit sphere $B_w$, i.e., $\mathbb{E}[\mathbf{u}\mathbf{u}^\top]=\frac{1}{d}I_{d \times d}$, by the definition of $\hat{\nabla} \ell(\mathbf{w}; \mathbf{z})$ in \eqref{smooth unige of loss function}, we have
		\begin{equation}
			\mathbb{E}_{\mathbf{u} \sim U_{B_w}}\left[\hat{\nabla} \ell(\mathbf{w}; \mathbf{z})\right] = \nabla \ell\left(\mathbf{w}; \mathbf{z}\right) + \frac{\tau_1^2\mu dL}{2} v(\mu,\mathbf{w}), \nonumber
		\end{equation}
		where $v(\mu,\mathbf{w})=\mathbb{E}_{\mathbf{u} \sim U_{B_w}}\left[\gamma(\mu,\mathbf{w},\mathbf{u})\mathbf{u}\right] \in \mathbb{R}^d$, and $\left\|v(\mu,\mathbf{z})\right\|_{w,*} \leq \tau_2\mathbb{E}_{\mathbf{u} \sim U_{B_w}}\left[\|\gamma(\mu,\mathbf{w},\mathbf{u})\|_2\|\mathbf{u}\|_2\right]$ $\leq \tau_2$, which completes the proof of \eqref{lem1.1}. 
		
		Then, we estimate the bound of $\left\| \hat{\nabla} \ell(\mathbf{w}; \mathbf{z})\right\|_{w,*}^2$. By Assumption \ref{assum:3} and similar to the proof of Proposition 6.6 in \citep{gao2018information}, we have
		\begin{align}
			&\mathbb{E}_{\mathbf{u} \sim U_{B_w}} \left[ \left\| \hat{\nabla} \ell(\mathbf{w}; \mathbf{z})\right\|_{w,*}^2 \right] \nonumber \\
			=& \frac{1}{\alpha(d)} \int_{B_w} \left\| \frac{d}{\mu} [\ell(\mathbf{w} + \mu \mathbf{u};\mathbf{z}) - \ell(\mathbf{w};\mathbf{z})]\mathbf{u} \right\|_{w,*}^2d\mathbf{u} \nonumber\\
			\leq& \frac{d^2}{\alpha(d) \mu^2} \int_{B_w} \Big\| \left[\ell(\mathbf{w} + \mu \mathbf{u};\mathbf{z}) - \ell(\mathbf{w};\mathbf{z}) - \langle \nabla \ell(\mathbf{w};\mathbf{z}), \mu \mathbf{u} \rangle + \langle \nabla \ell(\mathbf{w};\mathbf{z}), \mu \mathbf{u} \rangle \right]\mathbf{u} \Big\|_{w,*}^2 d\mathbf{u} \nonumber\\
			\leq& \frac{d^2}{\alpha(d) \mu^2} \int_{B_w} \Bigg[2\Big\| \left[\ell(\mathbf{w} + \mu \mathbf{u};\mathbf{z}) - \ell(\mathbf{w};\mathbf{z}) - \langle \nabla \ell(\mathbf{w};\mathbf{z}), \mu \mathbf{u} \rangle\right]\mathbf{u}\Big\|_{w,*} + 2\Big\|\langle \nabla \ell(\mathbf{w};\mathbf{z}), \mu \mathbf{u}\rangle\mathbf{u} \Big\|_{w,*}^2\Bigg] d\mathbf{u} \nonumber\\
			\leq& \frac{d^2}{\alpha(d) \mu^2}\left[\int_{B_w}2 \left\| \frac{L \mu^2}{2} \|\mathbf{u}\|_w^2\mathbf{u} \right\|_{w,*}^2 d\mathbf{u}+ \int_{B_w}2 \mu^2 \Big\|\|\nabla \ell(\mathbf{w};\mathbf{z})\|_{w,*}\|\mathbf{u}\|_{w}\mathbf{u}\Big\|_{w,*}^2\right] d\mathbf{u} \nonumber\\
			\leq &\frac{d^2}{\alpha(d) \mu^2} \left[ \frac{\tau_1^4\tau_2^2L^2 \mu^4}{2} \alpha(d) + 2 \mu^2 \tau_1^2\tau_2^2\|\nabla \ell(\mathbf{w};\mathbf{z})\|_{w,*}^2\alpha(d) \right] \nonumber\\
			= &2\tau_1^2\tau_2^2d^2\|\nabla \ell(\mathbf{w};\mathbf{z})\|_{w,*}^2 + \frac{\tau_1^4\tau_2^2\mu^2 d^2 L^2}{2},\nonumber
		\end{align}
		where $\alpha(d)$ is the surface area of $B_w$.
		By \eqref{expected gradient bound} and \eqref{variance bound}, we get
		\begin{align}
			\mathbb{E}_{\mathbf{z}\sim \mathcal{P}_i}\left[\left\| \nabla \ell(\mathbf{w}; \mathbf{z})\right\|_{w,*}^2\right] &\leq2\mathbb{E}_{\mathbf{z}\sim \mathcal{P}_i}\left[\left\| \nabla R_i(\mathbf{w}\right\|_{w,*}^2\right]+2\mathbb{E}_{\mathbf{z}\sim \mathcal{P}_i}\left[\left\| \nabla \ell(\mathbf{w}; \mathbf{z}) - \nabla R_i(\mathbf{w}) \right\|_{w,*}^2\right]\nonumber \\
			&\leq2G^2+2\sigma^2, \nonumber
		\end{align}
		which completes the proof of \eqref{lem1.2}.
	\end{proof}
	
	\subsubsection{Complexity Analysis}\label{3.1.2}
	\noindent Next, we analyze the iteration complexity of Algorithm \ref{alg:1}. To this end, we first present the optimization error of $\bar{\mathbf{w}}_t^{(i)}$ in \eqref{weighted value for smooth min} for each risk function.
	
	\begin{theorem}\label{thm:2}
		If Assumptions \ref{assum:1},\ref{assum:2},\ref{assum:3} and \ref{assum:4} hold, by setting $\mu_t=\frac{2}{\tau_1L\sqrt{t+1}}$ and $\eta_t^{(i)} = \frac{1}{\sqrt{2}\tau_1\tau_2d\sqrt{t+1}}$ in Algorithm \ref{alg:1}, we have
		\begin{equation}\label{smooth min convergence}
			\mathbb{E}[R_i(\bar{\mathbf{w}}_t^{(i)}) - R_i^*] \leq \frac{4\sqrt{2}d\tau_1\tau_2\left[D^2+\left(G^2+\sigma^2+D\right)\ln 3+\frac{3}{8}\right]}{\sqrt{t+2}}= \mathcal{O}\left(\frac{1}{\sqrt{t}}\right).
		\end{equation}
	\end{theorem}
	
	\begin{proof}
		Let $\mathbf{w}_{*}^{(i)} \in {\operatorname{argmin}}_{\mathbf{w} \in \mathcal{W}} R_i(\mathbf{w})$ be the optimal solution that minimizes $R_i(\mathbf{w})$. From the convexity of the risk function, we have
		\begin{align}\label{bias of smooth min}
			R_i(\bar{\mathbf{w}}_t^{(i)}) - R_i(\mathbf{w}_*^{(i)}) &= R_i\left(\sum_{j=\lceil\frac{t}{2}\rceil}^{t}\frac{\eta_j^{(i)} \mathbf{w}^{(i)}_j}{\sum_{k=\lceil\frac{t}{2}\rceil}^{t} \eta_k^{(i)}}\right) - R_i(\mathbf{w}_*^{(i)}) \nonumber \\
			&\leq \left(\sum_{j=\lceil\frac{t}{2}\rceil}^{t} \frac{\eta_j^{(i)}}{\sum_{k=\lceil\frac{t}{2}\rceil}^{t} \eta_k^{(i)}} R_i(\mathbf{w}_j^{(i)})\right) - R_i(\mathbf{w}_*^{(i)}) \nonumber \\
			&= \sum_{j=\lceil\frac{t}{2}\rceil}^{t} \frac{1}{\sum_{k=\lceil\frac{t}{2}\rceil}^{t} \eta_k^{(i)}}\eta_j^{(i)} \left(R_i(\mathbf{w}_j^{(i)}) - R_i(\mathbf{w}_*^{(i)})\right) \nonumber \\
			&\leq \sum_{j=\lceil\frac{t}{2}\rceil}^{t} \frac{1}{\sum_{k=\lceil\frac{t}{2}\rceil}^{t} \eta_k^{(i)}}\eta_j^{(i)} \left\langle \nabla R_i(\mathbf{w}_j^{(i)}), \mathbf{w}_j^{(i)} - \mathbf{w}_*^{(i)} \right\rangle.
		\end{align}
        The r.h.s of \eqref{bias of smooth min} can be rewritten as follows:
		\begin{align}\label{smooth min split}
			\eta_j^{(i)} \left\langle \nabla R_i(\mathbf{w}_j^{(i)}), \mathbf{w}_j^{(i)} - \mathbf{w}_*^{(i)} \right\rangle &= \eta_j^{(i)} \left\langle \hat{\nabla}\ell(\mathbf{w}_j^{(i)}; \mathcal{P}_i), \mathbf{w}_j^{(i)} - \mathbf{w}_*^{(i)}\right\rangle  \nonumber \\
			&\quad + \eta_j^{(i)} \left\langle \nabla R_i(\mathbf{w}_j^{(i)}) - \hat{\nabla}\ell(\mathbf{w}_j^{(i)}; \mathcal{P}_i), \mathbf{w}_j^{(i)} - \mathbf{w}_*^{(i)} \right\rangle.
		\end{align}
		Firstly, we estimate the first term in the r.h.s of \eqref{smooth min split}. From the property of mirror descent (Lemma 2.1 in \citep{nemirovski2009robust}), we have
		\begin{align}\label{smooth min split1}
			&\eta_j^{(i)} \left\langle\hat{\nabla}\ell(\mathbf{w}_j^{(i)}; \mathcal{P}_i), \mathbf{w}_j^{(i)} - \mathbf{w}_*^{(i)}\right\rangle \nonumber \\
			\leq& B_w\left(\mathbf{w}_*^{(i)}, \mathbf{w}_j^{(i)}\right) - B_w\left(\mathbf{w}_*^{(i)}, \mathbf{w}_{j+1}^{(i)}\right) + \frac{(\eta_j^{(i)})^2}{2} \left\|\hat{\nabla}\ell(\mathbf{w}_j^{(i)}; \mathcal{P}_i)\right\|_{w,*}^2.
		\end{align}
		Summing \eqref{smooth min split1} over $j = \lceil\frac{t}{2}\rceil,\cdots,t$ and taking expectation over both sides, we have
		\begin{align}\label{smooth min split1*}
			&\mathbb{E}\left[\sum_{j=\lceil\frac{t}{2}\rceil}^{t} \eta_j^{(i)} \left\langle\hat{\nabla}\ell(\mathbf{w}_j^{(i)}; \mathcal{P}_i), \mathbf{w}_j^{(i)} - \mathbf{w}_*^{(i)}\right\rangle \right] \nonumber\\
			\leq& B_w\left(\mathbf{w}_*^{(i)}, \mathbf{w}_1^{(i)}\right) +  \sum_{j=\lceil\frac{t}{2}\rceil}^{t}\frac{\left(\eta_j^{(i)}\right)^2}{2} \mathbb{E}\left[ \left\| \hat{\nabla}\ell(\mathbf{w}_j^{(i)}; \mathcal{P}_i) \right\|^2_{w,*}\right].
		\end{align}
		By \eqref{lem1.2} in Lemma \ref{lem:1}, we obtain
		\begin{align}\label{smooth min zo bound}
			&\mathbb{E} \left[\left\|\hat{\nabla}\ell(\mathbf{w}_j^{(i)}; \mathcal{P}_i)\right\|_{w,*}^2 \right] \leq \mathbb{E} \left[\left\| \frac{1}{r} \sum_{k=1}^{r} \hat{\nabla}\ell(\mathbf{w}_j^{(i)}; \mathbf{z}_{j,i}^{(k)})\right\|_{w,*}^2 \right]  \nonumber \\
			\leq& \frac{1}{r} \mathbb{E} \left[\sum_{k=1}^{r} \left\|\hat{\nabla}\ell(\mathbf{w}_j^{(i)}; \mathbf{z}_{j,i}^{(k)})\right\|_{w,*}^2 \right]\leq 4\tau_1^2\tau_2^2(G^2+\sigma^2)d^2 + \frac{\tau_1^4\tau_2^2 d^2 L^2}{2}\mu_j^2. 
		\end{align}
		Plugging \eqref{B_w bound} and \eqref{smooth min zo bound} into \eqref{smooth min split1*}, we then obtain
		\begin{align}\label{smooth min split1**}
			&\mathbb{E}\left[\sum_{j=\lceil\frac{t}{2}\rceil}^{t} \eta_j^{(i)} \left\langle\hat{\nabla}\ell(\mathbf{w}_j^{(i)}; \mathcal{P}_i), \mathbf{w}_j^{(i)} - \mathbf{w}_*^{(i)}\right\rangle \right] \nonumber\\
			\leq& D^2 + 2\tau_1^2\tau_2^2(G^2+\sigma^2)d^2\sum_{j=\lceil\frac{t}{2}\rceil}^{t} \left(\eta_j^{(i)}\right)^2 + \frac{\tau_1^4\tau_2^2 d^2 L^2}{4}\sum_{j=\lceil\frac{t}{2}\rceil}^{t}\mu_j^2 \left(\eta_j^{(i)}\right)^2.
		\end{align}
		Next, we estimate the second term in the r.h.s of \eqref{smooth min split}. By \eqref{lem1.1} in Lemma \ref{lem:1}, we have
		\begin{align} 
			\mathbb{E}_{(\mathbf{u},\mathbf{z})}\left[\hat{\nabla} \ell(\mathbf{w}_j^{(i)}; \mathcal{P}_i)\right]-\nabla R_i\left(\mathbf{w}_j^{(i)}\right)=\frac{\tau_1^2\mu_j dL}{2}v(\mu_j,\mathbf{w}_j^{(i)}).\nonumber
		\end{align}
		Then, from the property of Bregman distance and \eqref{B_w bound}, we get
		\begin{align}\label{smooth min split2}
			&\mathbb{E}\left[\sum_{j=\lceil\frac{t}{2}\rceil}^{t} \eta_j^{(i)} \left\langle \nabla R_i(\mathbf{w}_j^{(i)}) - \hat{\nabla}\ell(\mathbf{w}_j^{(i)}; \mathcal{P}_i), \mathbf{w}_j^{(i)} - \mathbf{w}_*^{(i)} \right\rangle\right] \nonumber\\ =&\sum_{j=\lceil\frac{t}{2}\rceil}^{t} \eta_j^{(i)} \left\langle -\frac{\tau_1^2\mu_j dL}{2}v(\mu_j,\mathbf{w}_j^{(i)}), \mathbf{w}_j^{(i)} - \mathbf{w}_*^{(i)} \right\rangle \nonumber \\
			\leq& \sum_{j=\lceil\frac{t}{2}\rceil}^{t} \frac{\tau_1^2\mu_j dL}{2}\eta_j^{(i)}\left\|v(\mu_j,\mathbf{w}_j^{(i)})\right\|_{w,*}\sqrt{2B_w\left(\mathbf{w}_j^{(i)},\mathbf{w}_*^{(i)}\right)} \nonumber\\
			\leq&\frac{\sqrt{2}\tau_1^2\tau_2DdL}{2}\sum_{j=\lceil\frac{t}{2}\rceil}^{t}\mu_j\eta_j^{(i)}.
		\end{align}
		Taking expectation over \eqref{bias of smooth min} and plugging \eqref{smooth min split1**} and \eqref{smooth min split2} into \eqref{bias of smooth min}, we have
		\begin{align}\label{smooth min simplify}
			&\mathbb{E}\left[R_i(\bar{\mathbf{w}}_t^{(i)}) - R_i(\mathbf{w}_*^{(i)})\right]  \nonumber\\
			\leq& \frac{D^2 + 2d^2\tau_1^2\tau_2^2(G^2+\sigma^2)\sum\limits_{j=\lceil\frac{t}{2}\rceil}^{t} \left(\eta_j^{(i)}\right)^2 + \frac{\tau_1^4\tau_2^2 d^2 L^2}{4}\sum\limits_{j=\lceil\frac{t}{2}\rceil}^{t}\mu_j^2 \left(\eta_j^{(i)}\right)^2+\frac{\sqrt{2}\tau_1^2\tau_2DdL}{2}\sum\limits_{j=\lceil\frac{t}{2}\rceil}^{t}\mu_j\eta_j^{(i)}}{\sum\limits_{j=\lceil\frac{t}{2}\rceil}^{t} \eta_j^{(i)}}.
		\end{align}
		For any $t \geq 2$, it holds that
		\begin{equation}\label{inequality1}
			\sum_{j=\lceil\frac{t}{2}\rceil}^{t} \frac{1}{(j+1)^2} \leq \frac{\pi^2}{6}-1\leq \frac{3}{4},
		\end{equation}
	then by setting $\mu_j=\frac{2}{\tau_1L\sqrt{j+1}}$ and $\eta_j^{(i)} = \frac{1}{\sqrt{2}\tau_1\tau_2d\sqrt{j+1}}$, we obtain 
		\begin{align}
			\sum_{j=\lceil\frac{t}{2}\rceil}^{t} \frac{1}{j+1} &\leq \int_{\lceil\frac{t}{2}\rceil-1}^{t} \frac{1}{x+1}dx= \ln\left(\frac{t+1}{\lceil0.5t\rceil}\right)\leq\ln\left(\frac{t+1}{0.5t}\right)=\ln\left(2+\frac{2}{t}\right)\leq\ln3,\label{inequality2}\\
			\sum_{j=\lceil\frac{t}{2}\rceil}^{t} \frac{1}{\sqrt{j+1}} &\geq \int_{\lceil\frac{t}{2}\rceil}^{t+1} \frac{1}{\sqrt{x+1}}dx =\frac{t}{\sqrt{t+2}+\sqrt{0.5t+2}}\geq\frac{t}{2\sqrt{t+2}}\geq\frac{1}{4}\sqrt{t+2}.\label{inequality3}
		\end{align}
		Plugging \eqref{inequality1}, \eqref{inequality2} and \eqref{inequality3} into \eqref{smooth min simplify}, we have
		\begin{align}
			&\mathbb{E}\left[R_i(\bar{\mathbf{w}}_t^{(i)}) - R_i(\mathbf{w}_*^{(i)})\right] \nonumber\\
			\leq& \frac{D^2+\left(G^2+\sigma^2\right)\sum_{j=\lceil\frac{t}{2}\rceil}^{t}\frac{1}{j+1}+\frac{1}{2}\sum_{j=\lceil\frac{t}{2}\rceil}^{t}\frac{1}{(j+1)^2}+D\sum_{j=\lceil\frac{t}{2}\rceil}^{t}\frac{1}{j+1}}{\sum_{j=\lceil\frac{t}{2}\rceil}^{t}\frac{1}{\sqrt{2}\tau_1\tau_2d\sqrt{j+1}}}\nonumber\\
			\leq& \frac{4\sqrt{2}\tau_1\tau_2d\left[D^2+\left(G^2+\sigma^2+D\right)\ln 3+\frac{3}{8}\right]}{\sqrt{t+2}}= \mathcal{O}\left(\frac{1}{\sqrt{t}}\right),
		\end{align}
		which completes the proof.
	\end{proof}
	By Theorem \ref{thm:2}, with a varying step size, the iteration complexity of ZO-SMD to estimate the expected excess risk in \eqref{smooth min convergence} is of the order $\mathcal{O}\left(1/\sqrt{t}\right)$, which improves the two-point gradient estimates algorithm proposed in \citep{duchi2015optimal} by a $\log t$ factor. Note that the SMD algorithm proposed in \citep{zhang2023efficient} to solve (non-)smooth MERO is actually a first-order method, and the iteration complexity is $\tilde{\mathcal{O}}\left(1/\sqrt{t}\right)$. 
	
	Next, we examine the optimization error of $\bar{\mathbf{w}}_t$ and $\bar{\mathbf{q}}_t$ in Algorithm \ref{alg:1}.
	\begin{theorem}\label{thm:3}
		If Assumptions \ref{assum:1},\ref{assum:2} and \ref{assum:3} hold, by setting
		\begin{equation}
			\eta_t^{(i)} = \frac{1}{\sqrt{2}\tau_1\tau_2d\sqrt{t+1}}, \quad \eta_t^{w} = \frac{2D^2}{\sqrt{2}\tau_1\tau_2d\sqrt{t+1}}, \quad \eta_t^{q} = \frac{2\ln m}{\sqrt{2}\tau_1\tau_2d\sqrt{t+1}} \nonumber
		\end{equation}
		in Algorithm \ref{alg:1}, and $\mu_t=\frac{2}{\tau_1L\sqrt{t+1}}$, we have
		\begin{align}\label{smooth minimax convergence}
			& \mathbb{E}\left[\max_{\mathbf{q} \in \Delta_m} \phi(\bar{\mathbf{w}}_t,\mathbf{q}) - \min_{\mathbf{w} \in \mathcal{W}} \phi(\mathbf{w}, \bar{\mathbf{q}}_t)\right] \nonumber\\
			\leq& \frac{4\sqrt{2}\tau_1\tau_2d}{\sqrt{t + 2}}\left\{2+\frac{15D^2}{4}+ \left\{D+\frac{20\tau_1^2\tau_2^2(G^2+\sigma^2)D^2d^2 + 5C^2 \ln m}{2\tau_1^2\tau_2^2d^2} \right.\right. \nonumber \\
			&\left.\left. +8\left[D^2+\left(G^2+\sigma^2+D\right)\ln 3+\frac{3}{8}\right]\right\}\ln 3\right\}=\mathcal{O}\left( \frac{1}{\sqrt{t}} \right).
		\end{align}
	\end{theorem}
	
	\begin{proof}
		We will first combine the two update rules in \eqref{ZO-SMD w for smooth minimax} and \eqref{ZO-SMD q for smooth minimax} into a single one according to the analysis of \citep{nemirovski2009robust}.
		
		\noindent (1) Merging the two update rules in \eqref{ZO-SMD w for smooth minimax} and \eqref{ZO-SMD q for smooth minimax}
		
		\noindent Let $\mathcal{E}$ be the space in which $\mathcal{W}$ resides. We equip the Cartesian product $\mathcal{E}\times\mathbb{R}^m$ with the following norm and dual norm:
		\begin{equation}
			\|(\mathbf{w}, \mathbf{q})\| = \sqrt{\frac{1}{2D^2} \|\mathbf{w}\|_w^2 + \frac{1}{2\ln m} \|\mathbf{q}\|_1^2},\nonumber
		\end{equation}
		\begin{equation}
			\|(\mathbf{u}, \mathbf{v})\|_* = \sqrt{2D^2 \|\mathbf{u}\|_{w,*}^2 + 2\ln m \|\mathbf{v}\|_\infty^2}.\nonumber
		\end{equation}
		We use the notation $\mathbf{x}=(\mathbf{w},\mathbf{q})$ , and equip the set $\mathcal{W}\times\Delta_m$ with the distance-generating function
		\begin{equation}
			\nu(\mathbf{x}) = \nu(\mathbf{w},\mathbf{q})=\frac{1}{2D^2} \nu_w(\mathbf{w}) + \frac{1}{2\ln m} \nu_q(\mathbf{q}).\nonumber
		\end{equation}
		It is easy  to verify that $\nu(\mathbf{x})$ is 1-strongly convex w.r.t. the norm $\|\cdot\|$. Let $B(\cdot,\cdot)$ be the Bregman distance associated with $\nu(\cdot)$:
		\begin{align}
			& B(\mathbf{x}, \mathbf{x}^{'}) = \nu(\mathbf{x}) - \left[\nu(\mathbf{x}^{'}) + \langle \nabla \nu(\mathbf{x}^{'}), \mathbf{x} - \mathbf{x}^{'} \rangle \right] \nonumber \\
			=& \frac{1}{2D^2} \nu_w(\mathbf{w}^{'}) - \left[\nu_w(\mathbf{w}^{'}) + \langle \nabla\nu_w(\mathbf{w}^{'}), \mathbf{w} - \mathbf{w}^{'} \rangle \right] + \frac{1}{2\ln m} \nu_q(\mathbf{q}) - \left[\nu_q(\mathbf{q}^{'}) + \langle \nabla \nu_q(\mathbf{q}^{'}), \mathbf{q} - \mathbf{q}^{'} \rangle \right] \nonumber \\
			=& \frac{1}{2D^2} B_w(\mathbf{w}, \mathbf{w}^{'}) + \frac{1}{2\ln m} B_q(\mathbf{q}, \mathbf{q}^{'}), \nonumber
		\end{align}
		where $\mathbf{x}^{'}=(\mathbf{w}^{'},\mathbf{q}^{'})$, and we have $(\mathbf{o}_w, \mathbf{o}_q) = {\operatorname{argmin}}_{(\mathbf{w},\mathbf{q})\in \mathcal{W}\times\Delta_m} \nu(\mathbf{w},\mathbf{q})$. Then by \eqref{B_w bound} and \eqref{B_q bound}, we can show that the domain $\mathcal{W}\times\Delta_m$ is bounded since
		\begin{equation}\label{B bound}
			\max_{(\mathbf{w},\mathbf{q})\in \mathcal{W}\times\Delta_m} B_q([\mathbf{w},\mathbf{q}], [\mathbf{o}_w, \mathbf{o}_q]) \leq \frac{1}{2D^2} \max_{\mathbf{w}\in \mathcal{W}}B_w(\mathbf{w}, \mathbf{o}_w) + \frac{1}{2\ln m} \max_{\mathbf{q}\in\Delta_m}B_q(\mathbf{q}, \mathbf{o}_q) \leq 1.
		\end{equation}
		With the above configurations, \eqref{ZO-SMD w for smooth minimax} and \eqref{ZO-SMD q for smooth minimax} are equivalent to
		\begin{equation}
			\mathbf{x}_{t+1} = \underset{\mathbf{x} \in \mathcal{W} \times \Delta_m}{\operatorname{argmin}}\left\{ \eta_t \left\langle \left[\hat{\mathbf{g}}_w(\mathbf{w}_t, \mathbf{q}_t), - \mathbf{g}_q(\mathbf{w}_t, \mathbf{q}_t)\right], \mathbf{x} - \mathbf{x}_t \right\rangle + B(\mathbf{x}, \mathbf{x}_t) \right\}. \nonumber
		\end{equation}
		(2) Analysis of ZO-SMD on smooth MERO
		
 To simplify the notation, we define the true gradient of $\phi(\mathbf{w},\mathbf{q})$ at $(\mathbf{w}_t,\mathbf{q}_t)$ as
		\begin{align}
			& F(\mathbf{w}_t, \mathbf{q}_t) = \left[ \nabla_w \phi(\mathbf{w}_t, \mathbf{q}_t),  - \nabla_q \phi(\mathbf{w}_t, \mathbf{q}_t) \right] \nonumber \\
			=& \left[ \sum_{i=1}^{m} q_{t,i} \nabla R_i(\mathbf{w}_t),-\left[ R_1(\mathbf{w}_t) - R_1^*, \ldots, R_m(\mathbf{w}_t) - R_m^* \right]^{\top} \right]. \nonumber
		\end{align}
		The zeroth-order stochastic gradient $\hat{\mathbf{g}}(\mathbf{w}_t, \mathbf{q}_t)$ is defined as:
		\begin{align}
			&\hat{\mathbf{g}}(\mathbf{w}_t, \mathbf{q}_t) = \left[ \hat{\mathbf{g}}_w(\mathbf{w}_t, \mathbf{q}_t), - \mathbf{g}_q(\mathbf{w}_t, \mathbf{q}_t) \right] \nonumber \\
			=& \left[ \sum_{i=1}^{m} q_{t,i} \hat{\nabla} \ell(\mathbf{w}_t; \mathcal{P}_i),-\left[\frac{1}{r} \sum_{j=1}^{r}\left[\ell(\mathbf{w}_t; \mathbf{z}_{t,1}^{(j)}) - \ell(\bar{\mathbf{w}}_t^{(1)},\mathbf{z}_{t,1}^{(j)})\right], \cdots, \frac{1}{r} \sum_{j=1}^{r}\left[\ell(\mathbf{w}_t; \mathbf{z}_{t,m}^{(j)}) - \ell(\bar{\mathbf{w}}_t^{(m)},\mathbf{z}_{t,m}^{(j)})\right] \right]^\top \right]. \nonumber
		\end{align}
		And the bias of $\hat{\mathbf{g}}(\mathbf{w}_t, \mathbf{q}_t)$ is characterized by
		\begin{align}
			& F(\mathbf{w}_t, \mathbf{q}_t) - \mathbb{E}_{t-1}[\hat{\mathbf{g}}(\mathbf{w}_t, \mathbf{q}_t)] \nonumber \\
			=& \left[ \sum_{i=1}^{m} q_{t,i} \left[\nabla R_i(\mathbf{w}_t) - \mathbb{E}\left[\hat{\nabla} \ell(\mathbf{w}_t; \mathcal{P}_i)\right] \right],
			- \left[ R_1(\bar{\mathbf{w}}_t^{(1)}) - R_1^*, \cdots, R_m(\bar{\mathbf{w}}_t^{(m)}) - R_m^* \right]^{\top} \right]. \nonumber
		\end{align}
		From the convexity-concavity of $\phi(\mathbf{w},\mathbf{q})$, by \eqref{weighted value for smooth minimax} we have the formula of the optimization error similarly to \citep{zhang2023efficient}:
		\begin{align}\label{smooth minimax optimization error}
			&\max_{\mathbf{q} \in \Delta_m} \phi(\bar{\mathbf{w}}_t, \mathbf{q}) - \min_{\mathbf{w} \in \mathcal{W}} \phi(\mathbf{w}, \bar{\mathbf{q}}_t)\nonumber \\
			=& \max_{\mathbf{q} \in \Delta_m} \phi \left( \sum_{j=\lceil\frac{t}{2}\rceil}^{t}\frac{\eta_j^{w} \mathbf{w}_j}{\sum_{k=\lceil\frac{t}{2}\rceil}^{t} \eta_k^{w}}, \mathbf{q} \right) - \min_{\mathbf{w} \in \mathcal{W}} \phi \left( \mathbf{w}, \sum_{j=\lceil\frac{t}{2}\rceil}^{t}\frac{\eta_j^{q} \mathbf{q}_j}{\sum_{k=\lceil\frac{t}{2}\rceil}^{t} \eta_k^{q}}\right)  \nonumber \\
			=& \max_{\mathbf{q} \in \Delta_m} \phi \left( \sum_{j=\lceil\frac{t}{2}\rceil}^{t}\frac{\eta_j \mathbf{w}_j}{\sum_{k=\lceil\frac{t}{2}\rceil}^{t} \eta_k}, \mathbf{q} \right) - \min_{\mathbf{w} \in \mathcal{W}} \phi \left( \mathbf{w}, \sum_{j=\lceil\frac{t}{2}\rceil}^{t}\frac{\eta_j \mathbf{q}_j}{\sum_{k=\lceil\frac{t}{2}\rceil}^{t} \eta_k}\right)  \nonumber \\
			\leq& \left( \sum_{j=\lceil\frac{t}{2}\rceil}^{t} \eta_j \right)^{-1} \left[ \max_{\mathbf{q} \in \Delta_m} \sum_{j=\lceil\frac{t}{2}\rceil}^{t} \eta_j \phi(\mathbf{w}_j, \mathbf{q}) - \min_{\mathbf{w} \in \mathcal{W}} \sum_{j=\lceil\frac{t}{2}\rceil}^{t} \eta_j \phi(\mathbf{w}, \mathbf{q}_j) \right]  \nonumber \\
			\leq& \left( \sum_{j=\lceil\frac{t}{2}\rceil}^{t} \eta_j \right)^{-1} \left[ \max_{\mathbf{x} \in \mathcal{W} \times \Delta_m} \sum_{j=\lceil\frac{t}{2}\rceil}^{t} \eta_j \left[ \left\langle \nabla_w \phi(\mathbf{w}_j, \mathbf{q}_j), \mathbf{w}_j - \mathbf{w}\right \rangle  - \left\langle \nabla_q \phi(\mathbf{w}_j, \mathbf{q}_j), \mathbf{q}_j - \mathbf{q} \right\rangle \right] \right] \nonumber  \\
			=& \left( \sum_{j=\lceil\frac{t}{2}\rceil}^{t} \eta_j \right)^{-1} \max_{\mathbf{x} \in \mathcal{W} \times \Delta_m} \sum_{j=\lceil\frac{t}{2}\rceil}^{t} \eta_j \left\langle F(\mathbf{w}_j, \mathbf{q}_j), \mathbf{x}_j - \mathbf{x} \right\rangle.
		\end{align}
		As a result, we can decompose \eqref{smooth minimax optimization error} as follows:
		\begin{align}\label{smooth minimax split}
			&\max_{\mathbf{q} \in \Delta_m} \phi(\bar{\mathbf{w}}_t, \mathbf{q}) - \min_{\mathbf{w} \in \mathcal{W}} \phi(\mathbf{w}, \bar{\mathbf{q}}_t) \nonumber \\
			\leq& \left( \sum_{j=\lceil\frac{t}{2}\rceil}^{t} \eta_j \right)^{-1} \max_{\mathbf{x} \in \mathcal{W} \times \Delta_m} \sum_{j=\lceil\frac{t}{2}\rceil}^{t} \eta_j \left\langle\hat{\mathbf{g}} (\mathbf{w}_j, \mathbf{q}_j), \mathbf{x}_j - \mathbf{x} \right\rangle  \nonumber \\
			&+ \left( \sum_{j=\lceil\frac{t}{2}\rceil}^{t} \eta_j \right)^{-1} \max_{\mathbf{x} \in \mathcal{W} \times \Delta_m} \sum_{j=\lceil\frac{t}{2}\rceil}^{t} \eta_j \left\langle \mathbb{E}_{j-1}\left[\hat{\mathbf{g}}(\mathbf{w}_j, \mathbf{q}_j)\right] - \hat{\mathbf{g}}(\mathbf{w}_j, \mathbf{q}_j), \mathbf{x}_j - \mathbf{x} \right\rangle \nonumber \\
			&+ \left( \sum_{j=\lceil\frac{t}{2}\rceil}^{t} \eta_j \right)^{-1} \max_{\mathbf{x} \in \mathcal{W} \times \Delta_m} \sum_{j=\lceil\frac{t}{2}\rceil}^{t} \eta_j \left\langle F(\mathbf{w}_j, \mathbf{q}_j) - \mathbb{E}_{j-1}\left[\hat{\mathbf{g}}(\mathbf{w}_j, \mathbf{q}_j)\right], \mathbf{x}_j - \mathbf{x} \right\rangle.
		\end{align}
		Firstly, we estimate the first term in the r.h.s of \eqref{smooth minimax split}. From the property of mirror descent, we have
		\begin{equation}\label{smooth minimax split1}
			\eta_j \langle \hat{\mathbf{g}}(\mathbf{w}_j, \mathbf{q}_j), \mathbf{x}_j - \mathbf{x}\rangle \leq B(\mathbf{x}, \mathbf{x}_j) - B(\mathbf{x}, \mathbf{x}_{j+1}) + \frac{\eta_j^2}{2} \left\|\hat{\mathbf{g}}(\mathbf{w}_j, \mathbf{q}_j)\right\|_{*}^2.
		\end{equation}
		The norm of zeroth-order stochastic gradient $\hat{\mathbf{g}}(\mathbf{w}_j, \mathbf{q}_j)$ is well-bounded:
		\begin{align}\label{smooth minimax w zo bound}
			\mathbb{E}\left[\left\| \hat{\mathbf{g}}_w(\mathbf{w}_j, \mathbf{q}_j) \right\|_{w,*}^2\right] &= \mathbb{E} \left[\left\| \sum_{i=1}^{m} q_{j,i} \hat{\nabla} \ell(\mathbf{w}_j; \mathcal{P}_i) \right\|_{w,*}^2 \right]\leq \mathbb{E} \left[ \left(\sum_{i=1}^{m} q_{j,i} \left\|\hat{\nabla} \ell(\mathbf{w}_j; \mathcal{P}_i) \right\|_{w,*}\right)^2 \right] \nonumber \\
			&= \left(\sum_{i=1}^{m} q_{j,i}\mathbb{E}\left[\left\|\hat{\nabla} \ell(\mathbf{w}_j; \mathcal{P}_i) \right\|_{w,*}\right] \right)^2 \nonumber\\
			&\leq \left(\sum_{i=1}^{m} q_{j,i}\sqrt{4\tau_1^2\tau_2^2(G^2+\sigma^2)d^2 + \frac{\tau_1^4\tau_2^2 d^2 L^2}{2}\mu_j^2}\right)^2\nonumber \\
			&=4\tau_1^2\tau_2^2(G^2+\sigma^2)d^2 + \frac{\tau_1^4\tau_2^2 d^2 L^2}{2}\mu_j^2,
		\end{align}
		where the second equality is by the independence of distributions, and the last inequality is by \eqref{smooth min zo bound} and the property of expectation: $\left(\mathbb{E}\left[X\right]\right)^2 \leq \mathbb{E}\left[X^2\right]$. By \eqref{loss function bound} in Assumption \ref{assum:2}, we have
		\begin{equation}\label{smooth minimax q zo bound}
			\|\mathbf{g}_q(\mathbf{w}_j, \mathbf{q}_j)\|_{\infty} \leq C. 
		\end{equation}
		Combining \eqref{smooth minimax w zo bound} and \eqref{smooth minimax q zo bound}, we then have
		\begin{align}\label{smooth minimax zo bound}
			\mathbb{E}\left[\left\| \hat{\mathbf{g}}(\mathbf{w}_j, \mathbf{q}_j) \right\|_{*}^2 \right]&= \mathbb{E}\left[ 2D^2 \left\| \hat{\mathbf{g}}_w(\mathbf{w}_j, \mathbf{q}_j)\right\|_{w,*}^2 + 2 \ln m \left\| \mathbf{g}_q(\mathbf{w}_j, \mathbf{q}_j) \right\|_{\infty}^2 \right]\nonumber \\
			&\leq 2D^2\left[4\tau_1^2\tau_2^2(G^2+\sigma^2)d^2 + \frac{\tau_1^4\tau_2^2\mu_j^2 d^2 L^2}{2} \right] + 2C^2 \ln m \nonumber \\
			&= 8\tau_1^2\tau_2^2D^2(G^2+\sigma^2)d^2 + 2C^2 \ln m+ \tau_1^4\tau_2^2D^2 L^2d^2 \mu_j^2. 
		\end{align}
		Summing \eqref{smooth minimax split1} over $j = \lceil\frac{t}{2}\rceil,\cdots,t$, maximizing $\mathbf{x}$ and finally taking expectation over both sides, we have
		\begin{align}\label{smooth minimax split1*}
			&\mathbb{E} \left[ \max_{\mathbf{x} \in \mathcal{W} \times \Delta_m} \sum_{j=\lceil\frac{t}{2}\rceil}^{t} \eta_j \langle \hat{\mathbf{g}}(\mathbf{w}_j, \mathbf{q}_j), \mathbf{x}_j - \mathbf{x} \rangle \right] \nonumber \\
			\leq& \max_{\mathbf{x} \in \mathcal{W} \times \Delta_m} B(\mathbf{x}, \mathbf{x}_1) + \left[4\tau_1^2\tau_2^2(G^2+\sigma^2)D^2d^2 + C^2 \ln m\right] \sum_{j=\lceil\frac{t}{2}\rceil}^{t} \eta_j^2 + \frac{\tau_1^4\tau_2^2D^2L^2 d^2 }{2}\sum_{j=\lceil\frac{t}{2}\rceil}^{t}\mu_j^2 \eta_j^2 \nonumber \\
			\leq& 1 + \left[4\tau_1^2\tau_2^2(G^2+\sigma^2)D^2d^2 + C^2 \ln m \right]\sum_{j=\lceil\frac{t}{2}\rceil}^{t} \eta_j^2 + \frac{\tau_1^4\tau_2^2D^2 L^2d^2 }{2}\sum_{j=\lceil\frac{t}{2}\rceil}^{t}\mu_j^2 \eta_j^2.
		\end{align}
		where first inequality is by \eqref{smooth minimax zo bound} and the second inequality is by \eqref{B bound}.
		
		Next, we estimate the second term in the r.h.s of \eqref{smooth minimax split}. We make use of the ``ghost iterate" technique proposed by \citep{nemirovski2009robust} to solve this problem. Detailedly, we create a virtual sequence $\{\mathbf{y}_t\}$ by  performing ZO-SMD with $\mathbb{E}_{t-1}\left[\hat{\mathbf{g}}(\mathbf{w}_t, \mathbf{q}_t)\right] - \hat{\mathbf{g}}(\mathbf{w}_t, \mathbf{q}_t)$ as the gradient:
		\begin{equation}
			\mathbf{y}_{t+1} = \underset{\mathbf{x} \in \mathcal{W} \times \Delta_m}{\operatorname{argmin}} \left\{ \eta_t \left\langle\mathbb{E}_{t-1}\left[\hat{\mathbf{g}}(\mathbf{w}_t, \mathbf{q}_t)\right] - \hat{\mathbf{g}}(\mathbf{w}_t, \mathbf{q}_t), \mathbf{x} - \mathbf{y}_t \right\rangle + B(\mathbf{x}, \mathbf{y}_t) \right\},
		\end{equation}
		with $\mathbf{y}_1=\mathbf{x}_1$. Then we further decompose the error term as
		\begin{align}\label{smooth minimax split2}
			&\max_{\mathbf{x} \in \mathcal{W} \times \Delta_m} \sum_{j=\lceil\frac{t}{2}\rceil}^{t} \eta_j \left\langle \mathbb{E}_{j-1}[\hat{\mathbf{g}}(\mathbf{w}_j, \mathbf{q}_j)] - \hat{\mathbf{g}}(\mathbf{w}_j, \mathbf{q}_j), \mathbf{x}_j - \mathbf{x} \right\rangle \nonumber \\
			\leq& \max_{\mathbf{x} \in \mathcal{W} \times \Delta_m}\sum_{j=\lceil\frac{t}{2}\rceil}^{t} \eta_j \left\langle \mathbb{E}_{j-1}[\hat{\mathbf{g}}(\mathbf{w}_j, \mathbf{q}_j)] - \hat{\mathbf{g}}(\mathbf{w}_j, \mathbf{q}_j), \mathbf{y}_j - \mathbf{x} \right\rangle  \nonumber \\
			&+\sum_{j=\lceil\frac{t}{2}\rceil}^{t} \eta_j \left\langle \mathbb{E}_{j-1}[\hat{\mathbf{g}}(\mathbf{w}_j, \mathbf{q}_j)] - \hat{\mathbf{g}}(\mathbf{w}_j, \mathbf{q}_j), \mathbf{x}_j - \mathbf{y}_j \right\rangle.
		\end{align}
From the property of mirror descent, we have
		\begin{align}\label{smooth minimax split2.1}
			& \eta_j \left\langle \mathbb{E}_{j-1}[\hat{\mathbf{g}}(\mathbf{w}_j, \mathbf{q}_j)] - \mathbf{g}(\mathbf{w}_j, \mathbf{q}_j), \mathbf{y}_j - \mathbf{x} \right\rangle \nonumber \\
			\leq& B(\mathbf{x}, \mathbf{y}_j) - B(\mathbf{x}, \mathbf{y}_{j+1}) + \frac{\eta_j^2}{2} \left\| \mathbb{E}_{j-1}[\hat{\mathbf{g}}(\mathbf{w}_j, \mathbf{q}_j)] - \hat{\mathbf{g}}(\mathbf{w}_j, \mathbf{q}_j) \right\|_*^2.
		\end{align}
		According to the properties of expectation: $\left(\mathbb{E}\left[X\right]\right)^2 \leq \mathbb{E}\left[X^2\right]$ and $\|\mathbb{E}\left[X\right]\| \leq \mathbb{E}\left[\|X\|\right]$, and by \eqref{smooth minimax zo bound}, we have
		\begin{align}
			&\mathbb{E}\left[\left\| \mathbb{E}_{j-1}[\hat{\mathbf{g}}(\mathbf{w}_j, \mathbf{q}_j)] - \hat{\mathbf{g}}(\mathbf{w}_j,\mathbf{ q}_j) \right\|_*^2\right] \nonumber\\
			\leq& \mathbb{E}\left[2\left\| \mathbb{E}_{j-1}\left[\hat{\mathbf{g}}(\mathbf{w}_j, \mathbf{q}_j)\right] \right\|_*^2 + 2\left\| \hat{\mathbf{g}}(\mathbf{w}_j, \mathbf{q}_j) \right\|_*^2\right] \leq 4\mathbb{E}\left[\left\| \hat{\mathbf{g}}(\mathbf{w}_j, \mathbf{q}_j) \right\|_*^2\right] \nonumber\\
			\leq& 4\left[8\tau_1^2\tau_2^2(G^2+\sigma^2)D^2d^2 + 2C^2 \ln m+ \tau_1^4\tau_2^2D^2L^2 d^2 \mu_j^2\right],
		\end{align}
		Summing \eqref{smooth minimax split2.1} over $j = \lceil\frac{t}{2}\rceil,\cdots,t$, maximizing $\mathbf{x}$ and taking expectation over both sides, we have
		\begin{align}\label{smooth minimax split2.1*}
			&\mathbb{E}\left[\max_{\mathbf{x} \in \mathcal{W} \times \Delta_m}\sum_{j=\lceil\frac{t}{2}\rceil}^{t} \eta_j \left\langle \mathbb{E}_{j-1}[\hat{\mathbf{g}}(\mathbf{w}_j, \mathbf{q}_j)] - \mathbf{g}(\mathbf{w}_j, \mathbf{q}_j), \mathbf{x}_j - \mathbf{x}\right\rangle\right] \nonumber \\
			\leq& \max_{\mathbf{x} \in \mathcal{W} \times \Delta_m}B(\mathbf{x}, \mathbf{y}_j) + \frac{\eta_j^2}{2} \mathbb{E}\left[\left\| \mathbb{E}_{j-1}[\hat{\mathbf{g}}(\mathbf{w}_j, \mathbf{q}_j)] - \hat{\mathbf{g}}(\mathbf{w}_j, \mathbf{q}_j) \right\|_*^2\right] \nonumber \\
			\leq& 1+\left[16\tau_1^2\tau_2^2(G^2+\sigma^2)D^2d^2\eta_j^2 + 4C^2 \ln m\right]\sum_{j=\lceil\frac{t}{2}\rceil}^{t}\eta_j^2+2\tau_1^4\tau_2^2D^2L^2 d^2  \sum_{j=\lceil\frac{t}{2}\rceil}^{t}\mu_j^2 \eta_j^2.
		\end{align}
		On the other hand, we define
		\begin{equation}
			\delta_j=\eta_j \left\langle \mathbb{E}_{j-1}[\hat{\mathbf{g}}(\mathbf{w}_j, \mathbf{q}_j)] - \hat{\mathbf{g}}(\mathbf{w}_j, \mathbf{q}_j), \mathbf{x}_j - \mathbf{y}_j \right\rangle.\nonumber
		\end{equation}
		Since $\mathbf{x}_j$ and $\mathbf{y}_j$ are independent of the random samples used to construct $\hat{\mathbf{g}}(\mathbf{w}_j, \mathbf{q}_j)$, $\delta_{\lceil\frac{t}{2}\rceil},\cdots,\delta_t$ forms a martingale difference sequence, we have
		\begin{equation}\label{smooth minimax split2.2}
			\mathbb{E}\left[\sum_{j=\lceil\frac{t}{2}\rceil}^{t}\delta_j\right]=0.
		\end{equation}
		By plugging \eqref{smooth minimax split2.1*} and \eqref{smooth minimax split2.2} into \eqref{smooth minimax split2}, we have
		\begin{align}\label{smooth minimax split2*}
			& \mathbb{E}\left[ \max_{\mathbf{x} \in \mathcal{W} \times \Delta_m} \sum_{j=\lceil\frac{t}{2}\rceil}^{t} \eta_j \left\langle \mathbb{E}_{j-1}[\hat{\mathbf{g}}(\mathbf{w}_j, \mathbf{q}_j)] - \hat{\mathbf{g}}(\mathbf{w}_j, \mathbf{q}_j), \mathbf{x}_j - \mathbf{x}\right\rangle\right] \nonumber\\
			\leq& 1+\left[16\tau_1^2\tau_2^2(G^2+\sigma^2)D^2d^2\eta_j^2 + 4C^2 \ln m\right]\sum_{j=\lceil\frac{t}{2}\rceil}^{t}\eta_j^2+2\tau_1^4\tau_2^2D^2L^2 d^2  \sum_{j=\lceil\frac{t}{2}\rceil}^{t}\mu_j^2 \eta_j^2.
		\end{align}
	We proceed to decompose the optimization error as
		\begin{align}\label{smooth minimax split3}
			&\eta_j \left\langle F(\mathbf{w}_j, \mathbf{q}_j) - \mathbb{E}_{j-1} [\hat{\mathbf{g}}(\mathbf{w}_j, \mathbf{q}_j)], \mathbf{x}_j - \mathbf{x} \right\rangle \nonumber \\
			=& \eta_j \left\langle \left[ \sum_{i=1}^{m} q_{j,i} \left[\nabla R_i(\mathbf{w}_j) - \mathbb{E}_{j-1}\left[\hat{\nabla} \ell(\mathbf{w}_j; \mathcal{P}_i)\right] \right], -\left[ R_1(\bar{\mathbf{w}}_j^{(1)}) - R_1^*, \ldots, R_m(\bar{\mathbf{w}}_j^{(m)}) - R_m^* \right]^\top \right], \mathbf{x}_j - \mathbf{x}\right\rangle \nonumber \\
			=& \eta_j \left\langle \sum_{i=1}^{m} q_{j,i} \left[\nabla R_i(\mathbf{w}_j) - \mathbb{E}_{j-1}\left[\hat{\nabla} \ell(\mathbf{w}_j; \mathcal{P}_i)\right] \right], \mathbf{w}_j - \mathbf{w} \right\rangle \nonumber \\
			&- \eta_j \left\langle \left[ R_1(\bar{\mathbf{w}}_j^{(1)}) - R_1^*, \cdots, R_m(\bar{\mathbf{w}}_j^{(m)}) - R_m^* \right]^\top, \mathbf{q}_j - \mathbf{q} \right\rangle.
		\end{align}
		Firstly, summing the first term in the r.h.s of \eqref{smooth minimax split3} over $j = \lceil\frac{t}{2}\rceil,\cdots,t$, maximizing $\mathbf{w}$ and taking expectation over both sides,  by \eqref{smooth min split2} we have
		\begin{align}\label{smooth minimax split3.1}
			&\mathbb{E} \left[\max_{\mathbf{w} \in \mathcal{W}} \sum_{j=\lceil\frac{t}{2}\rceil}^{t} \eta_j \left\langle \sum_{i=1}^{m} q_{j,i} \left[ \nabla R_i(\mathbf{w}_j) - \mathbb{E}_{j-1}\left[\hat{\nabla} \ell(\mathbf{w}_j; \mathcal{P}_i)\right] \right], \mathbf{w}_j - \mathbf{w} \right\rangle \right] \nonumber \\
			\leq& \sum_{i=1}^{m}  \sum_{j=\lceil\frac{t}{2}\rceil}^{t} q_{j,i} \frac{\tau_1^2\mu_j dL}{2}\eta_j\left\|v(\mu_j,\mathbf{w}_j\right\|_{w,*}\max_{\mathbf{w} \in \mathcal{W}}\sqrt{2B_w\left(\mathbf{w}_j,\mathbf{w}\right)} \nonumber \\
			\leq& \sum_{i=1}^{m}  \left[ \frac{\sqrt{2}\tau_1^2\tau_2DdL}{2}\sum_{j=\lceil\frac{t}{2}\rceil}^{t}q_{j,i}\mu_j\eta_j \right]
			= \frac{\sqrt{2}\tau_1^2\tau_2DdL}{2}\sum_{j=\lceil\frac{t}{2}\rceil}^{t}\mu_j\eta_j.
		\end{align}
		Next, we estimate the second term in the r.h.s of \eqref{smooth minimax split3}. Recall the result in Theorem \ref{thm:2} , we deduce the common bound of $\mathbb{E} \left[R_i(\bar{\mathbf{w}}_j^{(i)}) - R_i^*\right]$ for all $i \sim [m]$. Thus, we have
		\begin{equation}
			 \mathbb{E} \left[\max_{i \in [m]} \left\{ R_i(\bar{\mathbf{w}}_j^{(i)}) - R_i^* \right\}\right] \leq \frac{4\sqrt{2}\tau_1\tau_2d\left[D^2+\left(G^2+\sigma^2+D\right)\ln 3+\frac{3}{8}\right]}{\sqrt{j+2}}. \nonumber
		\end{equation}
		Maximizing $\mathbf{q}$ and taking expectation over the second term in the r.h.s of \eqref{smooth minimax split3} and by the definition of the dual norm, we have
		\begin{align}\label{smooth minimax split3.2}
			&\mathbb{E} \left[\max_{\mathbf{q} \in \Delta_m} \sum_{j=1}^{t} -\eta_j \left\langle \left[ R_1(\bar{\mathbf{w}}_j^{(1)}) - R_1^*, \cdots, R_m(\bar{\mathbf{w}}_j^{(m)}) - R_m^* \right]^\top, \mathbf{q}_j - \mathbf{q} \right\rangle\right] \nonumber \\
			\leq& \sum_{j=\lceil\frac{t}{2}\rceil}^{t} \eta_j\mathbb{E} \left[\left\|\left[ R_1(\bar{\mathbf{w}}_j^{(1)}) - R_1^*, \cdots, R_m(\bar{\mathbf{w}}_j^{(m)}) - R_m^* \right]\right\|_{\infty}\right]\max_{\mathbf{q} \in \Delta_m}\|\mathbf{q}_j - \mathbf{q}\|_{1}  \nonumber \\
			\leq& \sum_{j=\lceil\frac{t}{2}\rceil}^{t} 2\eta_j \mathbb{E} \left[\max_{i \in [m]} \left\{ R_i(\bar{\mathbf{w}}_j^{(i)}) - R_i^* \right\}\right] \nonumber\\
			\leq& \sum_{j=\lceil\frac{t}{2}\rceil}^{t} \frac{8\sqrt{2}\tau_1\tau_2d\left[D^2+\left(G^2+\sigma^2+D\right)\ln 3+\frac{3}{8}\right]}{\sqrt{j+2}}\eta_j.
		\end{align}
		Then by combining \eqref{smooth minimax split3.1} and \eqref{smooth minimax split3.2}, we have
		\begin{align}\label{smooth minimax split3*}
			&\mathbb{E}\left[ \max_{\mathbf{x} \in \mathcal{W} \times \Delta_m} \sum_{j=\lceil\frac{t}{2}\rceil}^{t} \eta_j \left\langle F(\mathbf{w}_j, \mathbf{q}_j) - \mathbb{E}_{t-1}[\hat{\mathbf{g}}(\mathbf{w}_j, \mathbf{q}_j)], \mathbf{x}_j - \mathbf{x} \right\rangle \right]
			\nonumber \\
			\leq& \frac{\sqrt{2}\tau_1^2\tau_2DdL}{2}\sum_{j=\lceil\frac{t}{2}\rceil}^{t}\mu_j\eta_j+\sum_{j=\lceil\frac{t}{2}\rceil}^{t} \frac{8\sqrt{2}\tau_1\tau_2d\left[D^2+\left(G^2+\sigma^2+D\right)\ln 3+\frac{3}{8}\right]}{\sqrt{j+2}}\eta_j. 
		\end{align}
		Finally, taking expectation over \eqref{smooth minimax split} and plugging \eqref{smooth minimax split1*}, \eqref{smooth minimax split2*} and \eqref{smooth minimax split3*} into \eqref{smooth minimax split}, we have
		\begin{align}\label{smooth minimax simplify}
			&\mathbb{E}\left[ \max_{\mathbf{q} \in \Delta_m} \phi(\bar{\mathbf{w}}_t, \mathbf{q}) - \min_{\mathbf{w} \in \mathcal{W}} \phi(\mathbf{w}, \bar{\mathbf{q}}_t) \right] \nonumber \\
			\leq& \left( \sum_{j=\lceil\frac{t}{2}\rceil}^{t} \eta_j \right)^{-1} \left\{ 2 + \left[20\tau_1^2\tau_2^2(G^2+\sigma^2)D^2d^2 + 5C^2 \ln m \right]\sum_{j=\lceil\frac{t}{2}\rceil}^{t} \eta_j^2 + \frac{5\tau_1^4\tau_2^2D^2 L^2d^2 }{2}\sum_{j=\lceil\frac{t}{2}\rceil}^{t}\mu_j^2 \eta_j^2 \right.\nonumber\\
			&\left.+\frac{\sqrt{2}\tau_1^2\tau_2DdL}{2}\sum_{j=\lceil\frac{t}{2}\rceil}^{t}\mu_j\eta_j+\sum_{j=\lceil\frac{t}{2}\rceil}^{t} \frac{8\sqrt{2}\tau_1\tau_2d\left[D^2+\left(G^2+\sigma^2+D\right)\ln 3+\frac{3}{8}\right]}{\sqrt{j+2}}\eta_j \right\}.
		\end{align}
		By setting $\eta_j=\frac{1}{\sqrt{2}\tau_1\tau_2d\sqrt{j+1}}$ and $\mu_j=\frac{2}{\tau_1L\sqrt{j+1}}$, we can transform \eqref{smooth minimax simplify} into
		\begin{align}\label{smooth minimax simplify*}
			&\mathbb{E}\left[ \max_{\mathbf{q} \in \Delta_m} \phi(\bar{\mathbf{w}}_t, \mathbf{q}) - \min_{\mathbf{w} \in \mathcal{W}} \phi(\mathbf{w}, \bar{\mathbf{q}}_t) \right] \nonumber \\
			\leq& \left( \sum_{j=\lceil\frac{t}{2}\rceil}^{t} \frac{1}{\sqrt{2}\tau_1\tau_2d\sqrt{j+1}} \right)^{-1}
			\left\{ 2 + \frac{20\tau_1^2\tau_2^2(G^2+\sigma^2)D^2d^2 + 5C^2 \ln m}{2\tau_1^2\tau_2^2d^2} \sum_{j=\lceil\frac{t}{2}\rceil}^{t} \frac{1}{ j+1}\right.\nonumber \\ 
			&\left.+5D^2\sum_{j=\lceil\frac{t}{2}\rceil}^{t}\frac{1}{(j+1)^2}+ D\sum_{j=\lceil\frac{t}{2}\rceil}^{t}\frac{1}{j+1}+\sum_{j=\lceil\frac{t}{2}\rceil}^{t} \frac{8\left[D^2+\left(G^2+\sigma^2+D\right)\ln 3+\frac{3}{8}\right]}{\sqrt{(j+2)(j+1)}} \right\}.
		\end{align}
		By the inequality
		\begin{equation}\label{inequality4}
			\sqrt{(j+2)(j+1)} \geq j+1,
		\end{equation}
		and plugging \eqref{inequality1}, \eqref{inequality2} and \eqref{inequality3} into \eqref{smooth minimax simplify*}, we have
		\begin{align}
			&\mathbb{E}\left[ \max_{\mathbf{q} \in \Delta_m} \phi(\bar{\mathbf{w}}_t, \mathbf{q}) - \min_{\mathbf{w} \in \mathcal{W}} \phi(\mathbf{w}, \bar{\mathbf{q}}_t) \right] \nonumber \\
			\leq& \frac{4\sqrt{2}\tau_1\tau_2d}{\sqrt{t + 2}}\left\{2+\frac{15D^2}{4}+ \left\{D+\frac{20\tau_1^2\tau_2^2(G^2+\sigma^2)D^2d^2 + 5C^2 \ln m}{2\tau_1^2\tau_2^2d^2} \right.\right. \nonumber \\
			&\left.\left. +8\left[D^2+\left(G^2+\sigma^2+D\right)\ln 3+\frac{3}{8}\right]\right\}\ln 3\right\} \nonumber\\
			= & \mathcal{O}\left( \frac{1}{\sqrt{t}} \right),
		\end{align}
		which completes the proof.
	\end{proof}
	Theorem \ref{thm:3} demonstrates that by the partial averaging iterates, ZO-SMD converges at the order of $\mathcal{O}\left( 1/\sqrt{t}\right)$ in solving smooth MERO. Combining with the result of Theorem \ref{thm:2}, we can obtain the total complexity is bounded by $\mathcal{O}\left( 1/\sqrt{t}\right) \times \mathcal{O}\left( 1/\sqrt{t}\right) =\mathcal{O}\left( 1/t\right)$, which reaches the optimal convergence rate of zeroth-order algorithms for smooth saddle-point problems proved in \citep{sadiev2021zeroth}. 
	\subsection{ZO-SMD for Non-smooth MERO}\label{3.2}
	In this section, we propose a ZO-SMD algorithm to solve MERO under non-smooth case and analyze the corresponding iteration complexity. We will maintain the notations of $\bar{\mathbf{w}}_t^{(i)}$, $\bar{\mathbf{w}}_t$, $\bar{\mathbf{q}}_t$ and $\mathbf{g}_q(\mathbf{w}_t, \mathbf{q}_t)$ mentioned in Section \ref{3.1}. 
	
	Similarly, we firstly estimate the value of $R_i^*$ in \eqref{mero2} under non-smooth case by using the zeroth-order stochastic mirror descent (ZO-SMD) algorithm. By \eqref{non-smooth unige}, we substitute $\hat{\nabla}\ell(\mathbf{w}_t^{(i)}; \mathcal{P}_{i})$ in \eqref{smooth unige of loss function} for $\hat{\nabla}_{\text{ns}} \ell(\mathbf{w}_t^{(i)}; \mathcal{P}_{i})$, which is defined as
	\begin{align}\label{non-smooth unige of loss function}
		\hat{\nabla}_{\text{ns}} \ell(\mathbf{w}_t^{(i)}; \mathcal{P}_{i})& = \frac{1}{r} \sum_{j=1}^{r} \hat{\nabla}_{\text{ns}} \ell(\mathbf{w}_t^{(i)}; \mathbf{z}_{t,i}^{(j)}) \nonumber\\
		&= \frac{1}{r} \sum_{j=1}^{r} \frac{\ell(\mathbf{w}_t^{(i)} + \mu_{1,t} \mathbf{u}_j+\mu_{2,t} \mathbf{v}_j; \mathbf{z}_{t,i}^{(j)}) - \ell(\mathbf{w}_t^{(i)}+\mu_{1,t} \mathbf{u}_j; \mathbf{z}_{t,i}^{(j)})}{\mu_{2,t}} \mathbf{v}_j, 
	\end{align}
	for all $i \in [m]$, where $\mathbf{u}_j, \mathbf{v}_j \in \mathbb{R}^d$ are random vectors drawn from one of three pairs of distributions mentioned in Section \ref{2.2}, $\{\mu_{1,t}\}_{t=1}^{\infty}$ and $\{\mu_{2,t}\}_{t=1}^{\infty}$ are non-increasing sequences of positive
	smoothing parameters. The non-smooth ZO-SMD algorithm update for estimating $R_i^*$ for all $i \in [m]$ is given by
	\begin{equation}\label{ZO-SMD for non-smooth min}
		\mathbf{w}_{t+1}^{(i)} = \underset{\mathbf{w} \in \mathcal{W}}{\operatorname{argmin}} \left\{ \eta_t^{(i)} \left\langle \hat{\nabla}_{\textnormal{ns}} \ell(\mathbf{w}_t^{(i)}; \mathcal{P}_{i}), \mathbf{w} - \mathbf{w}_t^{(i)} \right\rangle + B_{w}(\mathbf{w}, \mathbf{w}_t^{(i)}) \right\}, 
	\end{equation}
	where $\eta_t^{(i)} > 0$ is the step size. Smilarly, the partial averaging value $R_i(\bar{\mathbf{w}}_t^{(i)})$ is used as an approximate solution for $\min_{\mathbf{w} \in \mathcal{W}} R_i(\mathbf{w})$, and $\bar{\mathbf{w}}_t^{(i)}$ is defined as:
	\begin{equation}\label{weighted value for non-smooth min}
		\bar{\mathbf{w}}_t^{(i)}= \sum_{j=\lceil\frac{t}{2}\rceil}^{t} \frac{\eta_j^{(i)} \mathbf{w}_j^{(i)}}{\sum_{k=\lceil\frac{t}{2}\rceil}^{t} \eta_k^{(i)}},
	\end{equation}
	where $\lceil\cdot\rceil$ is the ceiling function.
	
	We proceed to minimize \eqref{mero2} by non-smooth ZO-SMD. Similarly, we define the zeroth-order gradient of $\phi(\mathbf{w},\mathbf{q})$ w.r.t. $\mathbf{w}$ in the $t$-th round as
	\begin{align}\label{gradient_w for non-smooth minimax}
		\hat{\mathbf{g}}_{w}^*(\mathbf{w}_t, \mathbf{q}_t) &= \sum_{i=1}^{m} q_{t,i} \hat{\nabla}_{\text{ns}}\ell(\mathbf{w}_{t}; \mathcal{P}_i) \nonumber\\
		&= \sum_{i=1}^{m} q_{t,i}\frac{1}{r} \sum_{j=1}^{r} \frac{\ell(\mathbf{w}_t^{(i)} + \mu_{1,t} \mathbf{u}_j+\mu_{2,t} \mathbf{v}_j; \mathbf{z}_{t,i}^{(j)}) - \ell(\mathbf{w}_t^{(i)}+\mu_{1,t} \mathbf{u}_j; \mathbf{z}_{t,i}^{(j)})}{\mu_{2,t}} \mathbf{v}_j.
	\end{align}
	We maintain the stochastic gradient of $\hat{\phi}(\mathbf{w},\mathbf{q})$ w.r.t. $\mathbf{q}$ in the $t$-th round as:
	\begin{equation}\label{gradient_q for non-smooth minimax}
		\mathbf{g}_q(\mathbf{w}_t, \mathbf{q}_t) = \left[\frac{1}{r} \sum_{j=1}^{r}\left[\ell(\mathbf{w}_t; \mathbf{z}_{t,1}^{(j)}) - \ell(\bar{\mathbf{w}}_t^{(1)},\mathbf{z}_{t,1}^{(j)})\right], \cdots, \frac{1}{r} \sum_{j=1}^{r}\left[\ell(\mathbf{w}_t; \mathbf{z}_{t,m}^{(j)}) - \ell(\bar{\mathbf{w}}_t^{(m)},\mathbf{z}_{t,m}^{(j)})\right] \right]^\top.
	\end{equation}
	Finally, equipped with the zeroth-order gradient and the stochastic gradient in \eqref{gradient_w for non-smooth minimax} and \eqref{gradient_q for non-smooth minimax},  non-smooth ZO-SMD updates $\mathbf{w}_t$ and $\mathbf{q}_t$ as:
	\begin{align}
		\mathbf{w}_{t+1}&=\underset{\mathbf{w} \in \mathcal{W}}{\operatorname{argmin}}\left\{\eta_{t}^{w}\langle \hat{\mathbf{g}}_{w}^*(\mathbf{w}_t,\mathbf{q}_t),\mathbf{w}-\mathbf{w}_{t}\rangle + B_w(\mathbf{w}, \mathbf{w}_{t})\right\},\label{ZO-SMD w for non-smooth minimax}\\
		\mathbf{q}_{t+1}&=\underset{\mathbf{q} \in \Delta{m}}{\operatorname{argmin}}\left\{\eta_{t}^{q}\langle -\mathbf{g}_{q}(\mathbf{w}_t,\mathbf{q}_t),\mathbf{q}-\mathbf{q}_{t}\rangle + B_q(\mathbf{q}, \mathbf{q}_{t})\right\},\label{ZO-SMD q for non-smooth minimax}
	\end{align}
	where $\eta_{t}^{w} > 0$ and $\eta_{t}^{q} > 0$ are step sizes. We will maintain the partial averaging value of iterates:
	\begin{equation}\label{weighted value for non-smooth minimax}
		\bar{\mathbf{w}}_t = \sum_{j=\lceil\frac{t}{2}\rceil}^{t}\frac{\eta_j^{w} \mathbf{w}_j}{\sum_{k=\lceil\frac{t}{2}\rceil}^{t} \eta_k^{w}},\quad \bar{\mathbf{q}}_t = \sum_{j=\lceil\frac{t}{2}\rceil}^{t}\frac{\eta_j^{q} \mathbf{q}_j}{\sum_{k=\lceil\frac{t}{2}\rceil}^{t} \eta_k^{q}}.
	\end{equation}
The detailed algorithm is presented as in Algorithm \ref{alg:2}. 
	\begin{algorithm}[H]
		\caption{Zeroth-Order Stochastic Mirror Descent Algorithm: Non-smooth Case}\label{alg:2}
		\begin{algorithmic}[1]
			\State Initialize $\mathbf{w}_1 = \textbf{w}_1^{(1)} =\cdots= \mathbf{w}_1^{(m)} = {\operatorname{argmin}}_{\mathbf{w}\in \mathcal{W}} v_w(\mathbf{w})$, and $\mathbf{q}_1 = \frac{1}{m}\mathbf{1}_m \in \mathbb{R}^m$.
			\For{$t = 1$ to $T$}
			\For{$\forall$ $i \in [m]$}
			\State  Draw $r$ random samples $ \mathbf{z}_{t,i}^{(1)},\cdots, \mathbf{z}_{t,i}^{(r)}$ from $\mathcal{P}_i$.
			\EndFor
			\For{$\forall$ $i \in [m]$}
			\State Calculate $\hat{\nabla}_{\text{ns}} \ell (\mathbf{w}_t^{(i)}; \mathcal{P}_i)$ according to \eqref{non-smooth unige of loss function} and update $\mathbf{w}_t^{(i)}$ according to \eqref{ZO-SMD for non-smooth min}.
			\EndFor
			\For{$\forall$ $i \in [m]$}
			\State Calculate the weighted average $\bar{\mathbf{w}}_t^{(i)}$ in \eqref{weighted value for non-smooth min}.
			\EndFor
			\State Calculate $\hat{\mathbf{g}}_w(\mathbf{w}_t,\mathbf{q}_t)$ and $\mathbf{g}_q(\mathbf{w}_t,\mathbf{q}_t)$ according to \eqref{gradient_w for non-smooth minimax} and \eqref{gradient_q for non-smooth minimax}.
			\State Update $\mathbf{w}_t$ and $\mathbf{q}_t$ according to \eqref{ZO-SMD w for non-smooth minimax} and \eqref{ZO-SMD q for non-smooth minimax}.
			\State Calculate the weighted average $\bar{\mathbf{w}}_t$ and $\bar{\mathbf{q}}_t$ in \eqref{weighted value for non-smooth minimax}.
			\EndFor
		\end{algorithmic}
	\end{algorithm}
	\subsubsection{Technical Preparation}
	In this section, we analyze the iteration complexity of the zeroth-order stochasitc mirror descent (ZO-SMD) algorithm for solving non-smooth MERO. Instead of the Lipschitz continuous gradient, we make the following assumption for $\ell(\mathbf{w};\mathbf{z})$.
	\begin{assumption}\label{assum:5}
		For all $i \in [m]$, $\ell(\mathbf{w};\mathbf{z})$ is Lipschitz continuous w.r.t. the $\ell_2$-norm, i.e., there exists a constant $L^*>0$ such that for any $\mathbf{w}_1, \mathbf{w}_2\in \mathcal{W}$ and $\mathbf{z} \sim \mathcal{P}_i$, we have
		\begin{equation}
			\|\ell(\mathbf{w}_1;\mathbf{z})-\ell(\mathbf{w}_2;\mathbf{z})\|_2 \leq L^*\|\mathbf{w}_1-\mathbf{w}_2\|_2. \nonumber
		\end{equation}
		Moreover, $R_i(\mathbf{w})$ is Lipschitz continuous with constant $L^*$.
	\end{assumption}
By Assumption \ref{assum:5}, we can deduce that the $\ell_2$-norm of $\nabla\ell(\mathbf{w};\mathbf{z})$ is bounded by $L^*$.
	\begin{lemma}\label{lem:4}{\textnormal{\textbf{(Lemma 4 in \citep{duchi2015optimal})}}}
		Let the random vector $\mathbf{u}$ be distributed as $\mathcal{N}(0, I_{d \times d})$, uniformly on the $\ell_2$-ball of radius $\sqrt{d + 2}$, or uniformly on the $\ell_2$-sphere of radius $\sqrt{d}$. For any $k \in \mathbb{N}$, there is a constant $c_k$ (dependent only on $k$) such that
		\[
		\mathbb{E}\left[ \|\mathbf{u}\|_2^k \right] \leq c_k d^{k/2}.
		\]
		In all cases we have $\mathbb{E}[\mathbf{u}\mathbf{u}^\top] = I_{d \times d}$, and $c_k \leq 3$ for $k = 4$ and $c_k \leq \sqrt{3}$ for $k = 3$.
	\end{lemma}
	Lemma \ref{lem:4} combines the three cases mentioned in the definition of non-smooth directional gradient estimator in Section \ref{2.2} into one, providing a unified upper bound, thereby eliminating the need for subsequent classification discussions in the following analysis.
	\begin{lemma}\label{lem:5}{\textnormal{\textbf{(Lemma 2 in \citep{duchi2015optimal})}}}
		 Denote ${R_i}_{\mu}(\mathbf{w})=\mathbb{E}_{\mathbf{u}}\left[R_i\left(\mathbf{w} + \mu \mathbf{u}\right)\right] $. If Assumption \ref{assum:5} holds, for all $i \in [m]$ and any $\mathbf{w} \in \mathcal{W}$, we have
		\begin{align}\label{lem5.1}
			\mathbb{E}\left[\hat{\nabla}_{\textnormal{ns}}\ell(\mathbf{w};\mathbf{z})\right] &= \nabla {R_i}_{\mu_1}(\mathbf{w}) + \frac{\mu_2}{\mu_1} L^* v(\mu_1, \mu_2, \mathbf{w}), 
		\end{align}
		where $v = v(\mu_1, \mu_2, \mathbf{w})$ is bounded by $\|v\|_2 \leq \frac{1}{2} \mathbb{E}[\|\mathbf{v}_i\|_2^3]$. There exists a constant $c$ such that
		\begin{align}\label{lem5.2}
			\mathbb{E}\left[ \left\|\hat{\nabla}_{\textnormal{ns}}\ell(\mathbf{w};\mathbf{z})\right\|_2^2 \right] &\leq c {L^*}^2 d \left( \sqrt{\frac{\mu_2}{\mu_1}} d + 1 + \log d \right).
		\end{align}
	\end{lemma}
	\subsubsection{Complexity Analysis}
	We analyze iteration complexity of Algorithm \ref{alg:2} in this subsection. To this end, we first present the optimization error of $\bar{\mathbf{w}}_t^{(i)}$ in \eqref{weighted value for non-smooth min} for each risk function.
	Unlike the proof of Theorem \ref{thm:2}, we introduce an additional smoothing parameter in the formula of non-smooth directional gradient estimator. Thus, we present a different approach to analyze convergence of non-smooth ZO-SMD, which will utilize the properties of smoothed function mentioned in \eqref{smooth f}.
	\begin{theorem}\label{thm:6}
		If Assumptions \ref{assum:1}, \ref{assum:2}, \ref{assum:4} and \ref{assum:5} hold, by setting $\eta_t^{(i)}=\frac{\sqrt{2}}{\tau_2L^*d\sqrt{t+1}}$, $\mu_{1,t}=\frac{1}{t+1}$ and $\mu_{2,t}=\frac{1}{d(t+1)^2}$ in Algorithm \ref{alg:2}, we have
		\begin{align}\label{non-smooth min convergence}
			&\mathbb{E}[R_i(\bar{\mathbf{w}}_t^{(i)}) - R_i^*] \nonumber \\
			\leq& \frac{2\sqrt{2}L^*\left[\tau_2D^2d+c\tau_2\sqrt{d}+c\tau_2\log(2d)\ln 3+\sqrt{3}\tau_2^2D\sqrt{d}+\sqrt{6}\sqrt{d}\right]}{\sqrt{t+2}} \nonumber \\
			=& \mathcal{O}\left(\frac{1}{\sqrt{t}}\right).
		\end{align}
	\end{theorem}
	\begin{proof}
		By \eqref{smooth f} and Lemma 9 in \cite{duchi2012randomized} and the definition of ${R_i}_{\mu}(\mathbf{w})$, i.e.,  ${R_i}_{\mu}(\mathbf{w})=\mathbb{E}_{(\mathbf{u},\mathbf{z})} \ell\left(\mathbf{w} + \mu \mathbf{u}; \mathbf{z}\right)$, it can be easily derived that $R_i(\mathbf{w}) \leq {R_i}_{\mu}(\mathbf{w}) \leq R_i(\mathbf{w}) + \mu L^* \sqrt{d + 2}$ for any $\mathbf{w} \in \mathcal{W}$. Noting that $\sqrt{d+2}<\sqrt{3d}$, and by the convexity of ${R_i}_{\mu}(\mathbf{w})$, we have
		\begin{align}\label{bias1*}
			&R_i(\bar{\mathbf{w}}_t^{(i)}) - R_i(\mathbf{w}_*^{(i)}) = R_i\left(\sum_{j=\lceil\frac{t}{2}\rceil}^{t}\frac{\eta_j^{(i)} \mathbf{w}^{(i)}_j}{\sum_{k=\lceil\frac{t}{2}\rceil}^{t} \eta_k^{(i)}}\right) - R_i(\mathbf{w}_*^{(i)}) \nonumber \\
			\leq& \sum_{j=\lceil\frac{t}{2}\rceil}^{t} \frac{1}{\sum_{k=\lceil\frac{t}{2}\rceil}^{t} \eta_k^{(i)}}\eta_j^{(i)} \left(R_i(\mathbf{w}_j^{(i)}) - R_i(\mathbf{w}_*^{(i)})\right) \nonumber \\
			\leq& \sum_{j=\lceil\frac{t}{2}\rceil}^{t} \frac{1}{\sum_{k=\lceil\frac{t}{2}\rceil}^{t} \eta_k^{(i)}}\eta_j^{(i)} \left({R_i}_{\mu_{1,j}}(\mathbf{w}_j^{(i)}) - {R_i}_{\mu_{1,j}}(\mathbf{w}_*^{(i)})\right)+\sqrt{3d}L^*\sum_{j=\lceil\frac{t}{2}\rceil}^{t} \frac{1}{\sum_{k=\lceil\frac{t}{2}\rceil}^{t} \eta_k^{(i)}}\eta_j^{(i)}\mu_{1,j} \nonumber \\
			\leq& \sum_{j=\lceil\frac{t}{2}\rceil}^{t} \frac{1}{\sum_{k=\lceil\frac{t}{2}\rceil}^{t} \eta_k^{(i)}}\eta_j^{(i)} \left\langle \nabla {R_i}_{\mu_{1,j}}(\mathbf{w}_j^{(i)}), \mathbf{w}_j^{(i)} - \mathbf{w}_*^{(i)}\right\rangle+\sqrt{3d}L^*\sum_{j=\lceil\frac{t}{2}\rceil}^{t} \frac{1}{\sum_{k=\lceil\frac{t}{2}\rceil}^{t} \eta_k^{(i)}}\eta_j^{(i)}\mu_{1,j}.
		\end{align}
		Similarly, we can split the first term in the r.h.s of \eqref{bias1*} into:
		\begin{align}\label{non-smooth min split}
			\eta_j^{(i)} \left\langle \nabla {R_i}_{\mu_{1,t}}(\mathbf{w}_j^{(i)}), \mathbf{w}_j^{(i)} - \mathbf{w}_*^{(i)} \right\rangle &= \eta_j^{(i)} \left\langle \hat{\nabla}_{\text{ns}} \ell(\mathbf{w}_j^{(i)}; \mathcal{P}_{i}), \mathbf{w}_j^{(i)} - \mathbf{w}_*^{(i)}\right\rangle  \nonumber \\
			&\quad + \eta_j^{(i)} \left\langle \nabla {R_i}_{\mu_{1,j}}(\mathbf{w}_j^{(i)}) - \hat{\nabla}_{\text{ns}} \ell(\mathbf{w}_j^{(i)}; \mathcal{P}_{i}), \mathbf{w}_j^{(i)} - \mathbf{w}_*^{(i)} \right\rangle.
		\end{align}
		Firstly, we estimate the first term of \eqref{non-smooth min split}. From the property of mirror descent, we have
		\begin{align}\label{non-smooth min split1}
			& \eta_j^{(i)} \left\langle\hat{\nabla}_{\text{ns}} \ell(\mathbf{w}_j^{(i)}; \mathcal{P}_{i}), \mathbf{w}_j^{(i)} - \mathbf{w}_*^{(i)}\right\rangle \nonumber \\
			\leq& B_w\left(\mathbf{w}_*^{(i)}, \mathbf{w}_j^{(i)}\right) - B_w\left(\mathbf{w}_*^{(i)}, \mathbf{w}_{j+1}^{(i)}\right) + \frac{(\eta_j^{(i)})^2}{2} \left\|\hat{\nabla}_{\text{ns}} \ell(\mathbf{w}_j^{(i)}; \mathcal{P}_{i})\right\|_{w,*}^2.
		\end{align}
		Summing \eqref{non-smooth min split1} over $j = \lceil\frac{t}{2}\rceil,\cdots,t$ and taking expectation over both sides, we get
		\begin{align}\label{non-smooth min split1*}
			&\mathbb{E}\left[\sum_{j=\lceil\frac{t}{2}\rceil}^{t} \eta_j^{(i)} \left\langle\hat{\nabla}_{\text{ns}} \ell(\mathbf{w}_j^{(i)}; \mathcal{P}_{i}), \mathbf{w}_j^{(i)} - \mathbf{w}_*^{(i)}\right\rangle \right] \nonumber\\
			\leq& B_w\left(\mathbf{w}_*^{(i)}, \mathbf{w}_1^{(i)}\right) +  \sum_{j=\lceil\frac{t}{2}\rceil}^{t}\frac{\left(\eta_j^{(i)}\right)^2}{2}  \mathbb{E}\left[\left\| \hat{\nabla}_{\text{ns}} \ell(\mathbf{w}_j^{(i)}; \mathcal{P}_{i}) \right\|^2_{w,*}\right].
		\end{align}
		Similar to the proof of Theorem \ref{lem:1}, by \eqref{lem5.2} in Lemma \ref{lem:5}, we have
		\begin{align}\label{non-smooth min zo bound}
			&\mathbb{E} \left[\left\|\hat{\nabla}_{\text{ns}}\ell(\mathbf{w}_j^{(i)}; \mathcal{P}_i)\right\|_{w,*}^2 \right] 
			\leq \frac{1}{r} \mathbb{E} \left[\sum_{k=1}^{r} \left\|\hat{\nabla}_{\text{ns}}\ell(\mathbf{w}_j^{(i)}; \mathbf{z}_{j,i}^{(k)})\right\|_{w,*}^2 \right] \nonumber \\
			\leq& c \tau_2^2{L^*}^2 d \left( \sqrt{\frac{\mu_{2,j}}{\mu_{1,j}}} d + 1 + \log d \right) \leq c \tau_2^2{L^*}^2 d\left( \sqrt{\frac{\mu_{2,j}}{\mu_{1,j}}} d + \log 2d \right). 
		\end{align}
		Plugging \eqref{B_w bound} and \eqref{non-smooth min zo bound} into \eqref{non-smooth min split1*}, we then have 
		\begin{align}\label{non-smooth min split1**}
			&\mathbb{E}\left[\sum_{j=\lceil\frac{t}{2}\rceil}^{t} \eta_j^{(i)} \left\langle\hat{\nabla}_{\text{ns}} \ell(\mathbf{w}_j^{(i)}; \mathcal{P}_{i}), \mathbf{w}_j^{(i)} - \mathbf{w}_*^{(i)}\right\rangle \right] \nonumber\\
			\leq& D^2 + \frac{c\tau_2^2 {L^*}^2 d}{2} \sum_{j=\lceil\frac{t}{2}\rceil}^{t}\left( \sqrt{\frac{\mu_{2,j}}{\mu_{1,j}}} d + \log 2d \right) \left(\eta_j^{(i)}\right)^2 \nonumber\\
			\leq& D^2 + \frac{c\tau_2^2 {L^*}^2 d^2}{2} \sum_{j=\lceil\frac{t}{2}\rceil}^{t} \sqrt{\frac{\mu_{2,j}}{\mu_{1,j}}}\left(\eta_j^{(i)}\right)^2+\frac{c \tau_2^2{L^*}^2 d\log(2d)}{2} \sum_{j=\lceil\frac{t}{2}\rceil}^{t} \left(\eta_j^{(i)}\right)^2.
		\end{align}
		Next, we estimate the second term of the r.h.s of \eqref{non-smooth min split}. By \eqref{lem5.1} in Lemma \ref{lem:5}, we obtain
		\begin{align}\label{non-smooth min split2}
			&\mathbb{E}\left[\eta_j^{(i)}\left\langle \nabla {R_i}_{\mu_{1,j}}(\mathbf{w}_j^{(i)}) - \hat{\nabla}_{\text{ns}} \ell(\mathbf{w}_j^{(i)}; \mathcal{P}_{i}), \mathbf{w}_j^{(i)} - \mathbf{w}_*^{(i)} \right\rangle\right]  \nonumber\\
			\leq& \frac{\mu_{2,j}}{\mu_{1,j}} L^*\left\|v(\mu_{1,j}, \mu_{2,j},\mathbf{w}_j^{(i)})\right\|_{w,*}\left\|\mathbf{w}_j^{(i)} - \mathbf{w}_*^{(i)}\right\|_w \nonumber \\
			\leq&\frac{\mu_{2,j}}{\mu_{1,j}} L^*\tau_2^2\left\|v(\mu_{1,j}, \mu_{2,j},\mathbf{w}_j^{(i)})\right\|_2\sqrt{2B_w\left(\mathbf{w}_j^{(i)},\mathbf{w}_*^{(i)}\right)} \nonumber \\
			\leq&\frac{\sqrt{2}}{2}\tau_2^2DL^* \mathbb{E}\left[\left\|\mathbf{v}_i\right\|_2^3\right]\frac{\mu_{2,j}}{\mu_{1,j}} .
		\end{align}
		Since Lemma \ref{lem:4} implies that $\mathbb{E}\left[\left\|\mathbf{v}_i\right\|_2^3\right] \leq \sqrt{3}d^{3/2}$, summing \eqref{non-smooth min split2} over $j = \lceil\frac{t}{2}\rceil,\cdots,t$, we have
		\begin{align}\label{non-smooth min split2*}
			&\mathbb{E}\left[\sum_{j=\lceil\frac{t}{2}\rceil}^{t} \eta_j^{(i)}\left\langle \nabla {R_i}_{\mu_{1,j}}(\mathbf{w}_j^{(i)}) - \hat{\nabla}_{\text{ns}} \ell(\mathbf{w}_j^{(i)}; \mathcal{P}_{i}), \mathbf{w}_j^{(i)} - \mathbf{w}_*^{(i)} \right\rangle\right]  \nonumber\\
			\leq& \frac{\sqrt{6}}{2}L^*Dd^{3/2}\tau_2^2 \sum_{j=\lceil\frac{t}{2}\rceil}^{t} \frac{\mu_{2,j}}{\mu_{1,j}}\eta_j^{(i)}.
		\end{align}
	Taking expection over \eqref{non-smooth min split}, and plugging \eqref{non-smooth min split1**} and \eqref{non-smooth min split2*} into \eqref{non-smooth min split}, we then obtain
		\begin{align}\label{non-smooth min simplify}
			&\mathbb{E}\left[R_i(\bar{\mathbf{w}}_t^{(i)}) - R_i(\mathbf{w}_*^{(i)})\right]  \nonumber\\
			\leq& \frac{D^2 + \frac{c {L^*}^2 d^2\tau_2^2}{2} \sum\limits_{j=\lceil\frac{t}{2}\rceil}^{t} \sqrt{\frac{\mu_{2,j}}{\mu_{1,j}}}\left(\eta_j^{(i)}\right)^2+\frac{c {L^*}^2 d\log(2d)\tau_2^2}{2} \sum\limits_{j=\lceil\frac{t}{2}\rceil}^{t} \left(\eta_j^{(i)}\right)^2+\frac{\sqrt{6}}{2}L^*Dd^{3/2}\tau_2^2 \sum\limits_{j=\lceil\frac{t}{2}\rceil}^{t} \frac{\mu_{2,j}}{\mu_{1,j}}\eta_j^{(i)}}{\sum\limits_{j=\lceil\frac{t}{2}\rceil}^{t} \eta_j^{(i)}} \nonumber\\
			&+\frac{\sqrt{3d}L^*\sum\limits_{j=\lceil\frac{t}{2}\rceil}^{t} \eta_j^{(i)}\mu_{1,j}}{\sum\limits_{j=\lceil\frac{t}{2}\rceil}^{t} \eta_j^{(i)}}.
		\end{align}
		By setting $\eta_j^{(i)}=\frac{\sqrt{2}}{\tau_2L^*d\sqrt{j+1}}$, $\mu_{1,j}=\frac{1}{j+1}$ and $\mu_{2,j}=\frac{1}{d(j+1)^2}$, plugging \eqref{inequality1}, \eqref{inequality2} and \eqref{inequality3} into \eqref{non-smooth min simplify}, and by the inequality $\sum_{j=\lceil\frac{t}{2}\rceil}^{t} \frac{1}{(j+1)^{3/2}} \leq 1$, we have
		\begin{align}
			&\mathbb{E}\left[R_i(\bar{\mathbf{w}}_t^{(i)}) - R_i(\mathbf{w}_*^{(i)})\right]  \nonumber\\
			\leq& \frac{D^2 +\frac{c}{\sqrt{d}} \sum\limits_{j=\lceil\frac{t}{2}\rceil}^{t} \frac{1}{(j+1)^{3/2}}+c\frac{\log (2d)}{d} \sum\limits_{j=\lceil\frac{t}{2}\rceil}^{t} \frac{1}{j+1}+\frac{\sqrt{3}\tau_2D}{\sqrt{d}}\sum\limits_{j=\lceil\frac{t}{2}\rceil}^{t} \frac{1}{(j+1)^{3/2}}+\frac{\sqrt{6}}{\tau_2\sqrt{d}}\sum\limits_{j=\lceil\frac{t}{2}\rceil}^{t} \frac{1}{(j+1)^{3/2}}}{\frac{\sqrt{2}}{\tau_2L^*d}\sum\limits_{j=\lceil\frac{t}{2}\rceil}^{t}\frac{1}{\sqrt{j+1}}} \nonumber \\
			\leq& \frac{2\sqrt{2}L^*\left[\tau_2D^2d+c\tau_2\sqrt{d}+c\tau_2\log(2d)\ln 3+\sqrt{3}\tau_2^2D\sqrt{d}+\sqrt{6}\sqrt{d}\right]}{\sqrt{t+2}} \nonumber \\
			=& \mathcal{O}\left(\frac{1}{\sqrt{t}}\right),
		\end{align}
		which completes the proof.
	\end{proof}
	With a varying step size, the expected convergence rate of ZO-SMD to solve non-smooth MERO is in the order of $\mathcal{O}\left(1/\sqrt{t}\right)$, which is faster than the SMD algorithm proposed in \citep{zhang2023efficient} and the two-point gradient estimates algorithm proposed in \citep{duchi2015optimal}, and reaches the $\mathcal{O}(1/\sqrt{t})$ bound \citep{nemirovski2009robust}. The  two-point gradient estimates algorithm (\citet{zhang2023efficient}) can also reach $\mathcal{O}\left(1/\sqrt{t}\right)$ under a slightly stronger condition that $\max_{\mathbf{w},\mathbf{w}^{'} \in \mathcal{W}}B_{w}(\mathbf{w},\mathbf{w}^{'})\leq D^{2}$.
	
	Next, we examine the optimization error of $\bar{\mathbf{w}}_t$ and $\bar{\mathbf{q}}_t$ in Algorithm \ref{alg:2}.
	Unlike the proof of Theorem \ref{thm:3}, we still need to utilize the properties of smoothed function in \eqref{smooth f} to address the convergence rate of the non-smooth MERO problem. However, this involves the difficulty of decomposing $\phi(\mathbf{w},\mathbf{q})$ into a formula containing ${R_i}_{\mu}(\mathbf{w})$, otherwise we will encounter dilemma in the estimate of item similarly to \eqref{smooth minimax split3.1}, which we elaborate on in detail below.
	\begin{theorem}\label{thm:7}
		Suppose Assumptions \ref{assum:1},\ref{assum:2}, \ref{assum:4} and \ref{assum:5} hold. Setting
		\begin{equation}
			\eta_t^{(i)} = \frac{\sqrt{2}}{\tau_2L^*d\sqrt{t+1}}, \quad \eta_t^{w} =\frac{2D^2}{\sqrt{2}\tau_1\tau_2L^*d\sqrt{t+1}}, \quad \eta_t^{q}=\frac{2\ln m}{\sqrt{2}\tau_1\tau_2L^*d\sqrt{t+1}} \nonumber
		\end{equation}
		in Algorithm \ref{alg:2}, and
		\begin{equation}
			\mu_{1,t}=\frac{1}{t+1},\quad \mu_{2,t}=\frac{1}{d(t+1)^2}, \nonumber
		\end{equation}
		we have
		\begin{align}\label{non-smooth minimax convergence}
			& \mathbb{E}\left[\max_{\mathbf{q} \in \Delta_m} \phi(\bar{\mathbf{w}}_t,\mathbf{q}) - \min_{\mathbf{w} \in \mathcal{W}} \phi(\mathbf{w}, \bar{\mathbf{q}}_t)\right] \nonumber\\
			\leq& \frac{2\sqrt{2}L^*}{\sqrt{t + 2}}\Bigg\{4\tau_1\tau_2d+\frac{5c\tau_2D^2\sqrt{d}}{\tau_1}+\sqrt{3}\tau_2^2D\sqrt{d}+\frac{5c\tau_2^2D^2 {L^*}^2 d \log (2d)+5C^2 \ln m}{\tau_1\tau_2{L^*}^2d}\ln3  \nonumber\\
			&+8\ln3\left[\tau_2D^2d+c\tau_2\sqrt{d}+c\tau_2\log(2d)\ln 3+\sqrt{3}\tau_2^2D\sqrt{d}+\sqrt{6}\sqrt{d}\right] \Bigg\} \nonumber \\
			=& \mathcal{O}\left( \frac{1}{\sqrt{t}} \right).
		\end{align}
	\end{theorem}
	\begin{proof}		
		Similar to the step (1) in the proof of Theorem \ref{thm:3}, we first combine the two update rules in \eqref{ZO-SMD w for non-smooth minimax} and \eqref{ZO-SMD q for non-smooth minimax} into a single one, i.e., \eqref{ZO-SMD w for non-smooth minimax} and \eqref{ZO-SMD q for non-smooth minimax} are equivalent to
		\begin{equation}
			\mathbf{x}_{t+1} = \underset{\mathbf{x} \in \mathcal{W} \times \Delta_m}{\operatorname{argmin}}\left\{ \eta_t \left\langle \left[\hat{\mathbf{g}}_w^*(\mathbf{w}_t, \mathbf{q}_t), - \mathbf{g}_q(\mathbf{w}_t, \mathbf{q}_t)\right], \mathbf{x} - \mathbf{x}_t \right\rangle + B(\mathbf{x}, \mathbf{x}_t) \right\}. \nonumber
		\end{equation}
		We skip some of repeated segments similar to the step (2) in the proof of Theorem \ref{thm:3} except the merged gradient and its norm. The non-smooth zeroth-order stochastic gradient $\hat{\mathbf{g}}^*(\mathbf{w}_t, \mathbf{q}_t)$ is defined as:
		\begin{align}
			& \hat{\mathbf{g}}^*(\mathbf{w}_t, \mathbf{q}_t) = \left[ \hat{\mathbf{g}}_w^*(\mathbf{w}_t, \mathbf{q}_t), - \mathbf{g}_q(\mathbf{w}_t, \mathbf{q}_t) \right] \nonumber \\
			=& \left[ \sum_{i=1}^{m} q_{t,i} \hat{\nabla}_{\text{ns}} \ell(\mathbf{w}_t; \mathcal{P}_i),-\left[\frac{1}{r} \sum_{j=1}^{r}\left[\ell(\mathbf{w}_t; \mathbf{z}_{t,1}^{(j)}) - \ell(\bar{\mathbf{w}}_t^{(1)},\mathbf{z}_{t,1}^{(j)})\right], \cdots, \frac{1}{r} \sum_{j=1}^{r}\left[\ell(\mathbf{w}_t; \mathbf{z}_{t,m}^{(j)}) - \ell(\bar{\mathbf{w}}_t^{(m)},\mathbf{z}_{t,m}^{(j)})\right] \right]^\top \right].\nonumber
		\end{align}
		By the proof of Theorem \ref{thm:6}, we have:
		\begin{align}\label{pre error}
			\phi(\mathbf{w}, \mathbf{q}_j)-\phi(\mathbf{w}_j, \mathbf{q}_j) &=\sum_{i=1}^{m}q_{j,i}\left[R_i(\mathbf{w})-R_i(\mathbf{w}_j)\right] \nonumber \\
			&\geq \sum_{i=1}^{m}q_{j,i}\left[{R_i}_{\mu_{1,j}}(\mathbf{w})-{R_i}_{\mu_{1,j}}(\mathbf{w}_j)\right]-\sqrt{3d}L^*\sum_{i=1}^{m}q_{j,i}\mu_{1,j} \nonumber \\
			&\geq \left\langle\sum_{i=1}^{m}q_{j,i}\nabla{R_i}_{\mu_{1,j}}(\mathbf{w}_j),\mathbf{w}-\mathbf{w}_j\right\rangle-\sqrt{3d}L^*\mu_{1,j}.
		\end{align}
		For simplicity, we denote $\nabla_w^* \phi(\mathbf{w}_j, \mathbf{q}_j)=\sum_{i=1}^{m}q_{j,i}\nabla{R_i}_{\mu_{1,j}}(\mathbf{w}_j)$ and define the nearly true gradient of $\phi(\mathbf{w},\mathbf{q})$ at $(\mathbf{w}_t,\mathbf{q}_t)$ as $F^*(\mathbf{w}_j, \mathbf{q}_j)=\left[\nabla_w^* \phi(\mathbf{w}_j, \mathbf{q}_j),-\nabla_q \phi(\mathbf{w}_j, \mathbf{q}_j)\right]$.
		Thus, from the convexity-concavity of $\phi(\mathbf{w},\mathbf{q})$ and \eqref{pre error}, we define the optimization error as
		\begin{align}\label{non-smooth minimax optimization error}
			&\max_{\mathbf{q} \in \Delta_m} \phi(\bar{\mathbf{w}}_t, \mathbf{q}) - \min_{\mathbf{w} \in \mathcal{W}} \phi(\mathbf{w}, \bar{\mathbf{q}}_t)\nonumber \\
			\leq& \left( \sum_{j=\lceil\frac{t}{2}\rceil}^{t} \eta_j \right)^{-1} \left[ \max_{\mathbf{q} \in \Delta_m} \sum_{j=\lceil\frac{t}{2}\rceil}^{t} \eta_j \phi(\mathbf{w}_j, \mathbf{q}) - \min_{\mathbf{w} \in \mathcal{W}} \sum_{j=\lceil\frac{t}{2}\rceil}^{t} \eta_j \phi(\mathbf{w}, \mathbf{q}_j) \right]  \nonumber \\
			\leq& \left( \sum_{j=\lceil\frac{t}{2}\rceil}^{t} \eta_j \right)^{-1} \max_{\mathbf{x} \in \mathcal{W} \times \Delta_m} \sum_{j=\lceil\frac{t}{2}\rceil}^{t} \eta_j \langle F^*(\mathbf{w}_j, \mathbf{q}_j), \mathbf{x}_j - \mathbf{x} \rangle +\left( \sum_{j=\lceil\frac{t}{2}\rceil}^{t} \eta_j \right)^{-1}\sqrt{3d}L^*\sum_{j=\lceil\frac{t}{2}\rceil}^{t} \mu_{1,j}\eta_j.
		\end{align}
		
		Similarly, \eqref{non-smooth minimax optimization error} can be decomposed as follows:
		\begin{align}\label{non-smooth minimax split}
			&\max_{\mathbf{q} \in \Delta_m} \phi(\bar{\mathbf{w}}_t, \mathbf{q}) - \min_{\mathbf{w} \in \mathcal{W}} \phi(\mathbf{w}, \bar{\mathbf{q}}_t) \nonumber \\
			\leq& \left( \sum_{j=\lceil\frac{t}{2}\rceil}^{t} \eta_j \right)^{-1} \max_{\mathbf{x} \in \mathcal{W} \times \Delta_m} \sum_{j=\lceil\frac{t}{2}\rceil}^{t} \eta_j \left\langle\hat{\mathbf{g}}^* (\mathbf{w}_j, \mathbf{q}_j), \mathbf{x}_j - \mathbf{x} \right\rangle \nonumber \\
			&+ \left( \sum_{j=\lceil\frac{t}{2}\rceil}^{t} \eta_j \right)^{-1} \max_{\mathbf{x} \in \mathcal{W} \times \Delta_m} \sum_{j=\lceil\frac{t}{2}\rceil}^{t} \eta_j \left\langle \mathbb{E}_{j-1}\left[\hat{\mathbf{g}}^*(\mathbf{w}_j, \mathbf{q}_j)\right] - \hat{\mathbf{g}}^*(\mathbf{w}_j, \mathbf{q}_j), \mathbf{x}_j - \mathbf{x} \right\rangle \nonumber \\
			&+ \left( \sum_{j=\lceil\frac{t}{2}\rceil}^{t} \eta_j \right)^{-1} \max_{\mathbf{x} \in \mathcal{W} \times \Delta_m} \sum_{j=\lceil\frac{t}{2}\rceil}^{t} \eta_j \left\langle F^*(\mathbf{w}_j, \mathbf{q}_j) - \mathbb{E}_{j-1}\left[\hat{\mathbf{g}}^*(\mathbf{w}_j, \mathbf{q}_j)\right], \mathbf{x}_j - \mathbf{x} \right\rangle \nonumber \\
			&+\left( \sum_{j=\lceil\frac{t}{2}\rceil}^{t} \eta_j \right)^{-1}\sqrt{3d}L^*\sum_{j=\lceil\frac{t}{2}\rceil}^{t} \mu_{1,j}\eta_j.
		\end{align}
		Firstly, we estimate the first term in the r.h.s of \eqref{non-smooth minimax split}. By the property of mirror descent, we have
		\begin{equation}\label{non-smooth minimax split1}
			\eta_j \left\langle \hat{\mathbf{g}}^*(\mathbf{w}_j, \mathbf{q}_j), \mathbf{x}_j - \mathbf{x}\right\rangle \leq B\left(\mathbf{x}, \mathbf{x}_j\right) - B\left(\mathbf{x}, \mathbf{x}_{j+1}\right) + \frac{\eta_j^2}{2} \left\|\hat{\mathbf{g}}^*(\mathbf{w}_j, \mathbf{q}_j)\right\|_{*}^2.
		\end{equation}
		The norm of zeroth-order stochastic gradient $\hat{\mathbf{g}}^*(\mathbf{w}_j, \mathbf{q}_j)$ is well-bounded by \eqref{non-smooth min zo bound}:
		\begin{align}\label{non-smooth minimax w zo bound}
			\mathbb{E}\left[\left\| \hat{\mathbf{g}}_w^*(\mathbf{w}_j, \mathbf{q}_j) \right\|_{w,*}^2\right] &= \mathbb{E} \left[\left\| \sum_{i=1}^{m} q_{j,i} \hat{\nabla}_{\text{ns}} \ell(\mathbf{w}_j; \mathcal{P}_i) \right\|_{w,*}^2 \right]\nonumber \\
			&\leq c {L^*}^2 d \tau_2^2\left( \sqrt{\frac{\mu_{2,j}}{\mu_{1,j}}} d + \log (2d) \right),
		\end{align}
		Then, we have
		\begin{align}\label{non-smooth minimax zo bound}
			\mathbb{E}\left[\left\| \hat{\mathbf{g}}^*(\mathbf{w}_j, \mathbf{q}_j) \right\|_{*}^2 \right]&= \mathbb{E}\left[ 2D^2 \left\| \hat{\mathbf{g}}_w^*(\mathbf{w}_j, \mathbf{q}_j)\right\|_{w,*}^2 + 2 \ln m \left\| \mathbf{g}_q(\mathbf{w}_j, \mathbf{q}_j) \right\|_{\infty}^2 \right]\nonumber \\
			&\leq 2c\tau_2^2D^2 {L^*}^2 d \log (2d)+2C^2 \ln m + 2c\tau_2^2D^2 {L^*}^2 d^2\sqrt{\frac{\mu_{2,j}}{\mu_{1,j}}}.
		\end{align}
		Plugging \eqref{non-smooth minimax zo bound} into \eqref{non-smooth minimax split1}, summing \eqref{non-smooth minimax split1} over $j = \lceil\frac{t}{2}\rceil,\cdots,t$, maximizing $\mathbf{x}$ and finally taking expectation over both sides, we have
		\begin{align}\label{non-smooth minimax split1*}
			&\mathbb{E} \left[ \max_{\mathbf{x} \in \mathcal{W} \times \Delta_m} \sum_{j=\lceil\frac{t}{2}\rceil}^{t} \eta_j \left\langle \hat{\mathbf{g}}^*(\mathbf{w}_j, \mathbf{q}_j), \mathbf{x}_j - \mathbf{x} \right\rangle \right] \leq\max_{\mathbf{x} \in \mathcal{W} \times \Delta_m}B\left(\mathbf{x}, \mathbf{x}_j\right) + \frac{\eta_j^2}{2} \mathbb{E}\left[\left\|\hat{\mathbf{g}}^*(\mathbf{w}_j, \mathbf{q}_j)\right\|_{*}^2\right] \nonumber \\
			\leq& 1 + \left[c\tau_2^2D^2 {L^*}^2 d \log (2d)+C^2 \ln m\right]\sum_{j=\lceil\frac{t}{2}\rceil}^{t} \eta_j^2 + c\tau_2^2D^2 {L^*}^2 d^2 \sum_{j=\lceil\frac{t}{2}\rceil}^{t} \sqrt{\frac{\mu_{2,j}}{\mu_{1,j}}}\eta_j^2.
		\end{align} 
		Next, we estimate the second term in the r.h.s of \eqref{non-smooth minimax split}. We also make use of the ``ghost iterate" technique proposed by \citep{nemirovski2009robust} to solve this problem. Detailedly, we create a virtual sequence $\{\mathbf{z}_t\}$ by  performing non-smooth ZO-SMD with $\mathbb{E}_{t-1}\left[\hat{\mathbf{g}}^*(\mathbf{w}_t, \mathbf{q}_t)\right] - \hat{\mathbf{g}}^*(\mathbf{w}_t, \mathbf{q}_t)$ as the gradient:
		\begin{equation}
			\mathbf{z}_{t+1} = \underset{\mathbf{x} \in \mathcal{W} \times \Delta_m}{\operatorname{argmin}} \left\{ \eta_t \left\langle\mathbb{E}_{t-1}\left[\hat{\mathbf{g}}^*(\mathbf{w}_t, \mathbf{q}_t)\right] - \hat{\mathbf{g}}^*(\mathbf{w}_t, \mathbf{q}_t), \mathbf{x} - \mathbf{z}_t \right\rangle + B(\mathbf{x}, \mathbf{z}_t) \right\},
		\end{equation}
		with $\mathbf{z}_1=\mathbf{x}_1$. Then we further decompose the error term as
		\begin{align}\label{non-smooth minimax split2}
			&\max_{\mathbf{x} \in \mathcal{W} \times \Delta_m} \sum_{j=\lceil\frac{t}{2}\rceil}^{t} \eta_j \left\langle \mathbb{E}_{j-1}[\hat{\mathbf{g}}^*(\mathbf{w}_j, \mathbf{q}_j)] - \hat{\mathbf{g}}^*(\mathbf{w}_j, \mathbf{q}_j), \mathbf{x}_j - \mathbf{x} \right\rangle \nonumber \\
			\leq& \max_{\mathbf{x} \in \mathcal{W} \times \Delta_m}\sum_{j=\lceil\frac{t}{2}\rceil}^{t} \eta_j \left\langle \mathbb{E}_{j-1}[\hat{\mathbf{g}}^*(\mathbf{w}_j, \mathbf{q}_j)] - \hat{\mathbf{g}}^*(\mathbf{w}_j, \mathbf{q}_j), \mathbf{z}_j - \mathbf{x} \right\rangle \nonumber \\
			&+\sum_{j=\lceil\frac{t}{2}\rceil}^{t} \eta_j \left\langle \mathbb{E}_{j-1}[\hat{\mathbf{g}}^*(\mathbf{w}_j, \mathbf{q}_j)] - \hat{\mathbf{g}}^*(\mathbf{w}_j, \mathbf{q}_j), \mathbf{x}_j - \mathbf{z}_j \right\rangle.
		\end{align}
	We estimate the first term in the r.h.s of \eqref{non-smooth minimax split2}. By the property of mirror descent, we have
		\begin{align}\label{non-smooth minimax split2.1}
			& \eta_j \left\langle \mathbb{E}_{j-1}[\hat{\mathbf{g}}^*(\mathbf{w}_j, \mathbf{q}_j)] - \hat{\mathbf{g}}^*(\mathbf{w}_j, \mathbf{q}_j), \mathbf{z}_j - \mathbf{x} \right\rangle \nonumber \\
			\leq& B\left(\mathbf{x}, \mathbf{z}_j\right) - B\left(\mathbf{x}, \mathbf{z}_{j+1}\right) + \frac{\eta_j^2}{2} \left\| \mathbb{E}_{j-1}[\hat{\mathbf{g}}^*(\mathbf{w}_j, \mathbf{q}_j)] - \hat{\mathbf{g}}^*(\mathbf{w}_j, \mathbf{q}_j) \right\|_*^2.
		\end{align}
		Since $\left(\mathbb{E}\left[X\right]\right)^2 \leq \mathbb{E}\left[X^2\right]$ and $\|\mathbb{E}\left[\cdot\right]\| \leq \mathbb{E}\left[\|\cdot\|\right]$, and by \eqref{non-smooth minimax zo bound}, we get
		\begin{align}
			&\mathbb{E}\left[\left\| \mathbb{E}_{j-1}[\hat{\mathbf{g}}^*(\mathbf{w}_j, \mathbf{q}_j)] - \hat{\mathbf{g}}^*(\mathbf{w}_j,\mathbf{ q}_j) \right\|_*^2\right] \leq \mathbb{E}\left[2\left\| \mathbb{E}_{j-1}\left[\hat{\mathbf{g}}^*(\mathbf{w}_j, \mathbf{q}_j)\right] \right\|_*^2 + 2\left\| \hat{\mathbf{g}}^*(\mathbf{w}_j, \mathbf{q}_j) \right\|_*^2\right] \nonumber \\
			\leq& 4\mathbb{E}\left[\left\| \hat{\mathbf{g}}^*(\mathbf{w}_j, \mathbf{q}_j) \right\|_*^2\right]\leq 8c\tau_2^2D^2 {L^*}^2 d \log (2d)+8C^2 \ln m + 8c\tau_2^2D^2 {L^*}^2 d^2\sqrt{\frac{\mu_{2,j}}{\mu_{1,j}}}.
		\end{align}
		Summing \eqref{non-smooth minimax split2.1} over $j = \lceil\frac{t}{2}\rceil,\cdots,t$, maximizing $\mathbf{x}$ and taking expectation over both sides, we have
		\begin{align}\label{non-smooth minimax split2.1*}
			&\mathbb{E}\left[\max_{\mathbf{x} \in \mathcal{W} \times \Delta_m}\sum_{j=\lceil\frac{t}{2}\rceil}^{t} \eta_j \langle \mathbb{E}_{j-1}[\hat{\mathbf{g}}^*(\mathbf{w}_j, \mathbf{q}_j)] - \hat{\mathbf{g}}^*(\mathbf{w}_j, \mathbf{q}_j), \mathbf{x}_j - \mathbf{x}\rangle\right] \nonumber \\
			\leq& \max_{\mathbf{x} \in \mathcal{W}} B\left(\mathbf{x}, \mathbf{z}_j\right) + \frac{\eta_j^2}{2} \mathbb{E}\left[\left\| \mathbb{E}_{j-1}[\hat{\mathbf{g}}^*(\mathbf{w}_j, \mathbf{q}_j)] - \hat{\mathbf{g}}^*(\mathbf{w}_j, \mathbf{q}_j) \right\|_*^2\right] \nonumber \\
			\leq& 1+ \left[4c\tau_2^2D^2 {L^*}^2 d \log (2d)+4C^2 \ln m\right]\sum_{j=\lceil\frac{t}{2}\rceil}^{t} \eta_j^2 + 4c\tau_2^2D^2 {L^*}^2 d^2 \sum_{j=\lceil\frac{t}{2}\rceil}^{t} \sqrt{\frac{\mu_{2,j}}{\mu_{1,j}}}\eta_j^2.
		\end{align}
		Then, we estimate the second term in the r.h.s of \eqref{non-smooth minimax split2.1}. We define
		\begin{equation}
			\delta_j^*=\eta_j \left\langle \mathbb{E}_{j-1}[\hat{\mathbf{g}}^*(\mathbf{w}_j, \mathbf{q}_j)] - \hat{\mathbf{g}}^*(\mathbf{w}_j, \mathbf{q}_j), \mathbf{x}_j - \mathbf{z}_j \right\rangle. \nonumber
		\end{equation}
		Since $\mathbf{x}_j$ and $\mathbf{z}_j$ are independent of the random samples used to construct $\hat{\mathbf{g}}^*(\mathbf{w}_j, \mathbf{q}_j)$, $\delta_{\lceil\frac{t}{2}\rceil}^*,\cdots,\delta_t^*$ forms a martingale difference sequence. Thus, we have.
		\begin{equation}\label{non-smooth minimax split2.2}
			\mathbb{E}\left[\sum_{j=\lceil\frac{t}{2}\rceil}^{t}\delta_j^*\right]=0.
		\end{equation}
		Taking expectation over \eqref{non-smooth minimax split2}, and plugging \eqref{non-smooth minimax split2.1*} and \eqref{non-smooth minimax split2.2} into \eqref{non-smooth minimax split2}, we obtain
		\begin{align}\label{non-smooth minimax split2*}
			& \mathbb{E}\left[ \max_{\mathbf{x} \in \mathcal{W} \times \Delta_m} \sum_{j=\lceil\frac{t}{2}\rceil}^{t} \eta_j \left\langle \mathbb{E}_{j-1}[\hat{\mathbf{g}}^*(\mathbf{w}_j, \mathbf{q}_j)] - \hat{\mathbf{g}}^*(\mathbf{w}_j, \mathbf{q}_j), \mathbf{x}_j - \mathbf{x}\right\rangle\right] \nonumber \\
			\leq& 1+ \left[4c\tau_2^2D^2 {L^*}^2 d \log (2d)+4C^2 \ln m\right]\sum_{j=\lceil\frac{t}{2}\rceil}^{t} \eta_j^2 + 4c\tau_2^2D^2 {L^*}^2 d^2 \sum_{j=\lceil\frac{t}{2}\rceil}^{t} \sqrt{\frac{\mu_{2,j}}{\mu_{1,j}}}\eta_j^2.
		\end{align}
		Finally, we estimate the third term in the r.h.s of \eqref{non-smooth minimax split}. We proceed to decompose the optimization error as
		\begin{align}\label{non-smooth minimax split3}
			&\eta_j \left\langle F^*(\mathbf{w}_j, \mathbf{q}_j) - \mathbb{E}_{j-1} [\hat{\mathbf{g}}^*(\mathbf{w}_j, \mathbf{q}_j)], \mathbf{x}_j - \mathbf{x} \right\rangle \nonumber \\
			=& \eta_j \left\langle \left[ \sum_{i=1}^{m} q_{j,i} \left[\nabla {R_i}_{\mu_1}(\mathbf{w}_j) - \mathbb{E}_{j-1}\left[\hat{\nabla}_{\text{ns}} \ell(\mathbf{w}_j; \mathcal{P}_i)\right] \right], -\left[ R_1(\bar{\mathbf{w}}_j^{(1)}) - R_1^*, \ldots, R_m(\bar{\mathbf{w}}_j^{(m)}) - R_m^* \right]^\top \right], \mathbf{x}_j - \mathbf{x}\right\rangle \nonumber \\
			=&\eta_j \left\langle \sum_{i=1}^{m} q_{j,i} \left[\nabla {R_i}_{\mu_1}(\mathbf{w}_j) -\mathbb{E}_{j-1} \left[\hat{\nabla}_{\text{ns}} \ell(\mathbf{w}_j; \mathcal{P}_i) \right]\right], \mathbf{w}_j - \mathbf{w} \right\rangle\nonumber \\
			&- \eta_j \left\langle \left[ R_1(\bar{\mathbf{w}}_j^{(1)}) - R_1^*, \ldots, R_m(\bar{\mathbf{w}}_j^{(m)}) - R_m^* \right]^\top, \mathbf{q}_j - \mathbf{q} \right\rangle.
		\end{align}
		Summing the first term in the r.h.s of \eqref{non-smooth minimax split3} over $j = \lceil\frac{t}{2}\rceil,\cdots,t$, maximizing $\mathbf{w}$ and taking expectation over both sides, by \eqref{non-smooth min split2*} we have
		\begin{align}\label{non-smooth minimax split3.1}
			&\mathbb{E} \left[\max_{\mathbf{w} \in \mathcal{W}} \sum_{j=\lceil\frac{t}{2}\rceil}^{t} \eta_j \left\langle \sum_{i=1}^{m} q_{j,i} \left[ \nabla {R_i}_{\mu_1}(\mathbf{w}_j) - \mathbb{E}_{j-1} \left[\hat{\nabla}_{\text{ns}} \ell(\mathbf{w}_j; \mathcal{P}_i) \right]\right], \mathbf{w}_j - \mathbf{w} \right\rangle \right] \nonumber \\
			=& \sum_{i=1}^{m} \sum_{j=\lceil\frac{t}{2}\rceil}^{t} q_{j,i}\frac{\mu_{2,j}}{\mu_{1,j}} L^*\tau_2^2\left\|v\left(\mu_{1,j}, \mu_{2,j}, \mathbf{w}_j\right)\right\|_2\max_{\mathbf{w} \in \mathcal{W}}\sqrt{2B_w\left(\mathbf{w}_j,\mathbf{w}\right)} \nonumber \\
			\leq& \sum_{i=1}^{m}  \left[ \frac{\sqrt{6}}{2}L^*Dd^{3/2}\tau_2^2 \sum_{j=\lceil\frac{t}{2}\rceil}^{t}q_{j,i} \frac{\mu_{2,j}}{\mu_{1,j}}\eta_j \right] \nonumber \\
			=& \frac{\sqrt{6}}{2}L^*Dd^{3/2}\tau_2^2 \sum_{j=\lceil\frac{t}{2}\rceil}^{t} \frac{\mu_{2,j}}{\mu_{1,j}}\eta_j.
		\end{align}
		Next, we estimate the second term in the r.h.s of \eqref{non-smooth minimax split3.1}. Recall the result in Theorem \ref{thm:6} , we deduce the common bound of $\mathbb{E}\left[R_i(\bar{\mathbf{w}}_j^{(i)}) - R_i^*\right]$ for all $i \sim [m]$. Thus, we have
		\begin{align}
			&\mathbb{E} \left[\max_{i \in [m]} \left\{ R_i(\bar{\mathbf{w}}_j^{(i)}) - R_i^* \right\}\right]  \nonumber\\
			\leq& \frac{2\sqrt{2}L^*\left[\tau_2D^2d+c\tau_2\sqrt{d}+c\tau_2\log(2d)\ln 3+\sqrt{3}\tau_2^2D\sqrt{d}+\sqrt{6}\sqrt{d}\right]}{\sqrt{j+2}}. 	\nonumber 		
		\end{align}
		By the definition of the dual norm, we obtain
		\begin{align}\label{non-smooth minimax split3.2}
			&\mathbb{E} \left[\max_{\mathbf{q} \in \Delta_m} \sum_{j=\lceil\frac{t}{2}\rceil}^{t} \eta_j \left\langle \left[ R_1(\bar{\mathbf{w}}_j^{(1)}) - R_1^*, \cdots, R_m(\bar{\mathbf{w}}_j^{(m)}) - R_m^* \right]^\top, \mathbf{q}_j - \mathbf{q} \right\rangle\right] \nonumber \\
			\leq& \sum_{j=\lceil\frac{t}{2}\rceil}^{t} \eta_j\mathbb{E} \left[\left\|\left[ R_1(\bar{\mathbf{w}}_j^{(1)}) - R_1^*, \cdots, R_m(\bar{\mathbf{w}}_j^{(m)}) - R_m^* \right]\right\|_{\infty}\right]\max_{\mathbf{q} \in \Delta_m}\|\mathbf{q}_j - \mathbf{q}\|_{1}\nonumber\\
			\leq& \sum_{j=\lceil\frac{t}{2}\rceil}^{t} 2\eta_j \mathbb{E} \left[\max_{i \in [m]} \left\{ R_i(\bar{\mathbf{w}}_j^{(i)}) - R_i^* \right\}\right] \nonumber\\
			\leq& \sum_{j=\lceil\frac{t}{2}\rceil}^{t}\frac{4\sqrt{2}L^*\left[\tau_2D^2d+c\tau_2\sqrt{d}+c\tau_2\log(2d)\ln 3+\sqrt{3}\tau_2^2D\sqrt{d}+\sqrt{6}\sqrt{d}\right]}{\sqrt{j+2}}\eta_j.
		\end{align}
		Then combining \eqref{non-smooth minimax split3.1} and \eqref{non-smooth minimax split3.2}, we get
		\begin{align}\label{non-smooth minimax split3*}
			&\mathbb{E}\left[ \max_{\mathbf{x} \in \mathcal{W} \times \Delta_m} \sum_{j=\lceil\frac{t}{2}\rceil}^{t} \eta_j \left\langle F^*(\mathbf{w}_j, \mathbf{q}_j) - \mathbb{E}_{t-1}[\hat{\mathbf{g}}^*(\mathbf{w}_j, \mathbf{q}_j)], \mathbf{x}_j - \mathbf{x} \right\rangle \right]
			\nonumber \\
			\leq& \frac{\sqrt{6}}{2}L^*Dd^{3/2}\tau_2^2 \sum_{j=\lceil\frac{t}{2}\rceil}^{t} \frac{\mu_{2,j}}{\mu_{1,j}}\eta_j \nonumber \\
			&+\sum_{j=\lceil\frac{t}{2}\rceil}^{t}\frac{4\sqrt{2}L^*\left[\tau_2D^2d+c\tau_2\sqrt{d}+c\tau_2\log(2d)\ln 3+\sqrt{3}\tau_2^2D\sqrt{d}+\sqrt{6}\sqrt{d}\right]}{\sqrt{j+2}}\eta_j. 
		\end{align}
		Finally, taking expectation over \eqref{non-smooth minimax split} and plugging \eqref{non-smooth minimax split1*}, \eqref{non-smooth minimax split2*} and \eqref{non-smooth minimax split3*} into \eqref{non-smooth minimax split}, we have
		\begin{align}\label{non-smooth minimax simplify}
			& \mathbb{E}\left[ \max_{\mathbf{q} \in \Delta_m} \phi(\bar{\mathbf{w}}_t, \mathbf{q}) - \min_{\mathbf{w} \in \mathcal{W}} \phi(\mathbf{w}, \bar{\mathbf{q}}_t) \right] \nonumber \\
			\leq& \left( \sum_{j=\lceil\frac{t}{2}\rceil}^{t} \eta_j \right)^{-1} \left\{ 2 + \left[5c\tau_2^2D^2 {L^*}^2 d \log (2d)+5C^2 \ln m\right]\sum_{j=\lceil\frac{t}{2}\rceil}^{t} \eta_j^2 + 5c\tau_2^2D^2 {L^*}^2 d^2 \sum_{j=\lceil\frac{t}{2}\rceil}^{t} \sqrt{\frac{\mu_{2,j}}{\mu_{1,j}}}\eta_j^2 \right.\nonumber\\
			&+\frac{\sqrt{6}}{2}L^*Dd^{3/2}\tau_2^2 \sum_{j=\lceil\frac{t}{2}\rceil}^{t} \frac{\mu_{2,j}}{\mu_{1,j}}\eta_j+\sqrt{3d}L^*\sum_{j=\lceil\frac{t}{2}\rceil}^{t} \mu_{1,j}\eta_j \nonumber \\
			&\left. +\sum_{j=\lceil\frac{t}{2}\rceil}^{t}\frac{4\sqrt{2}L^*\left[\tau_2D^2d+c\tau_2\sqrt{d}+c\tau_2\log(2d)\ln 3+\sqrt{3}\tau_2^2D\sqrt{d}+\sqrt{6}\sqrt{d}\right]}{\sqrt{j+2}}\eta_j\right\}.
		\end{align}
		By setting the step size $\eta_j=\frac{1}{\sqrt{2}\tau_1\tau_2L^*d\sqrt{j+1}}$ and recall that we set $\mu_{1,j}=\frac{1}{j+1}$ and $\mu_{2,j}=\frac{1}{d(j+1)^2}$ in Theorem \ref{thm:6}, we can transform \eqref{non-smooth minimax simplify} into
		\begin{align}\label{non-smooth minimax simplify*}
			& \mathbb{E}\left[ \max_{\mathbf{q} \in \Delta_m} \phi(\bar{\mathbf{w}}_t, \mathbf{q}) - \min_{\mathbf{w} \in \mathcal{W}} \phi(\mathbf{w}, \bar{\mathbf{q}}_t) \right] \nonumber \\
			\leq& \left( \sum_{j=\lceil\frac{t}{2}\rceil}^{t} \frac{1}{\sqrt{2}\tau_1\tau_2L^*d\sqrt{j+1}} \right)^{-1}
			\left\{ 2 + \frac{5c\tau_2^2D^2 {L^*}^2 d \log (2d)+5C^2 \ln m}{2\tau_1^2\tau_2^2{L^*}^2d^2} \sum_{j=\lceil\frac{t}{2}\rceil}^{t} \frac{1}{ j+1}\right.\nonumber \\ 
			&+ \frac{5cD^2}{2\tau_1^2\sqrt{d}}\sum_{j=\lceil\frac{t}{2}\rceil}^{t}\frac{1}{(j+1)^{3/2}}+\frac{\sqrt{3}\tau_2D}{2\tau_1\sqrt{d}} \sum_{j=\lceil\frac{t}{2}\rceil}^{t} \frac{1}{(j+1)^{3/2}}+\frac{\sqrt{3}}{\sqrt{2}\tau_1\tau_2\sqrt{d}}\sum_{j=\lceil\frac{t}{2}\rceil}^{t}\frac{1}{(j+1)^{3/2}}\nonumber \\
			&\left.+\frac{4\left[\tau_2D^2d+c\tau_2\sqrt{d}+c\tau_2\log(2d)\ln 3+\sqrt{3}\tau_2^2D\sqrt{d}+\sqrt{6}\sqrt{d}\right]}{\tau_1\tau_2d}\sum_{j=\lceil\frac{t}{2}\rceil}^{t}\frac{1}{j+1}\right\}.
		\end{align}
		Plugging \eqref{inequality1} and \eqref{inequality2} into \eqref{non-smooth minimax simplify*}, we have
		\begin{align}
			&\mathbb{E}\left[ \max_{\mathbf{q} \in \Delta_m} \phi(\bar{\mathbf{w}}_t, \mathbf{q}) - \min_{\mathbf{w} \in \mathcal{W}} \phi(\mathbf{w}, \bar{\mathbf{q}}_t) \right] \nonumber \\
			\leq& \frac{2\sqrt{2}L^*}{\sqrt{t + 2}}\Bigg\{4\tau_1\tau_2d+\frac{5c\tau_2D^2\sqrt{d}}{\tau_1}+\sqrt{3}\tau_2^2D\sqrt{d}+\frac{5c\tau_2^2D^2 {L^*}^2 d \log (2d)+5C^2 \ln m}{\tau_1\tau_2{L^*}^2d}\ln3  \nonumber\\
			&+8\ln3\left[\tau_2D^2d+c\tau_2\sqrt{d}+c\tau_2\log(2d)\ln 3+\sqrt{3}\tau_2^2D\sqrt{d}+\sqrt{6}\sqrt{d}\right] \Bigg\} \nonumber \\
			=& \mathcal{O}\left( \frac{1}{\sqrt{t}} \right),
		\end{align}
		which completes the proof.
	\end{proof}
	Theorem \ref{thm:7} demonstrates that ZO-SMD converges at the order of $\mathcal{O}\left( 1/\sqrt{t}\right)$ for solving non-smooth MERO, which improves the convergence rate of general SMD by a $\log t$ factor~\citep{zhang2023efficient}. Combining with Theorem \ref{thm:6}, we can obtain the total complexity is bounded by $\mathcal{O}\left( 1/\sqrt{t}\right) \times \mathcal{O}\left( 1/\sqrt{t}\right) =\mathcal{O}\left( 1/t\right)$, which reaches the optimal convergence rate of zeroth-order algorithms for non-smooth stochastic saddle-point problems proved in \citep{dvinskikh2022noisy}.
	
	\section{Conclusions and Future Work}
	\noindent This paper proposes a zeroth-order stochastic mirror descent (ZO-SMD) algorithm for solving MERO under both smooth and non-smooth cases. Detailedly, ZO-SMD is based on SMD by substituting smooth zeroth-order gradient of loss function for the original stochasic gradient, which solves gradient-unknown cases and attains optimal convergence rates of $\mathcal{O}\left(1/\sqrt{t}\right)$ on the estimate of $R_i^*$ and $\mathcal{O}\left(1/\sqrt{t}\right)$ on the optimization error of $\bar{\mathbf{w}}_t$ and $\bar{\mathbf{q}}_t$. Moreover, we inrtoduce another smoothing parameter into the construction of non-smooth zeroth-order gradient to solve non-smooth MERO, and finally attain the same convergence rate on the estimate of $R_i^*$ and the optimization error of $\bar{\mathbf{w}}_t$ and $\bar{\mathbf{q}}_t$ with smooth ZO-SMD. To the best of our knowledge, it is the first zeroth-order algorithm with complexity gurantee for solving both smooth and non-smooth MERO problems.
	
	Moreover, although there are also some algorithms fitting on solving MERO under nonconvex case \citep{xu2021derivative,shen2023zeroth,xu2023zeroth}, when the estimate of $R_i^*$ is nearly optimal. It is worthwhile to further study MERO in the future how to estimate the value of $R_i^*$ in a non-convex setting, that is, the loss function is non-convex with respect to $\mathbf{w}$.
%
	
	\vskip 0.2in
	\bibliography{JMLR}

\begin{thebibliography}{63}
\providecommand{\natexlab}[1]{#1}
\providecommand{\url}[1]{\texttt{#1}}
\expandafter\ifx\csname urlstyle\endcsname\relax
  \providecommand{\doi}[1]{doi: #1}\else
  \providecommand{\doi}{doi: \begingroup \urlstyle{rm}\Url}\fi

\bibitem[Agarwal and Zhang(2022)]{agarwal2022minimax}
Alekh Agarwal and Tong Zhang.
\newblock Minimax regret optimization for robust machine learning under
  distribution shift.
\newblock In \emph{Conference on Learning Theory}, pages 2704--2729. PMLR,
  2022.

\bibitem[Allen-Zhu(2017)]{allen2017natasha}
Zeyuan Allen-Zhu.
\newblock Natasha: Faster non-convex stochastic optimization via strongly
  non-convex parameter.
\newblock In \emph{International Conference on Machine Learning}, pages 89--97.
  PMLR, 2017.

\bibitem[Allen-Zhu and Hazan(2016)]{allen2016variance}
Zeyuan Allen-Zhu and Elad Hazan.
\newblock Variance reduction for faster non-convex optimization.
\newblock In \emph{International Conference on Machine Learning}, pages
  699--707. PMLR, 2016.

\bibitem[Babanezhad and Lacoste-Julien(2020)]{babanezhad2020geometry}
Reza Babanezhad and Simon Lacoste-Julien.
\newblock Geometry-aware universal mirror-prox.
\newblock \emph{arXiv preprint arXiv:2011.11203}, 2020.

\bibitem[Beck and Teboulle(2003)]{beck2003mirror}
Amir Beck and Marc Teboulle.
\newblock Mirror descent and nonlinear projected subgradient methods for convex
  optimization.
\newblock \emph{Operations Research Letters}, 31\penalty0 (3):\penalty0
  167--175, 2003.

\bibitem[Ben-Tal et~al.(2013)Ben-Tal, Den~Hertog, De~Waegenaere, Melenberg, and
  Rennen]{ben2013robust}
Aharon Ben-Tal, Dick Den~Hertog, Anja De~Waegenaere, Bertrand Melenberg, and
  Gijs Rennen.
\newblock Robust solutions of optimization problems affected by uncertain
  probabilities.
\newblock \emph{Management Science}, 59\penalty0 (2):\penalty0 341--357, 2013.

\bibitem[Ben-Tal et~al.(2015)Ben-Tal, Hazan, Koren, and Mannor]{ben2015oracle}
Aharon Ben-Tal, Elad Hazan, Tomer Koren, and Shie Mannor.
\newblock Oracle-based robust optimization via online learning.
\newblock \emph{Operations Research}, 63\penalty0 (3):\penalty0 628--638, 2015.

\bibitem[Bregman(1967)]{bregman1967relaxation}
Lev~M Bregman.
\newblock The relaxation method of finding the common point of convex sets and
  its application to the solution of problems in convex programming.
\newblock \emph{USSR Computational Mathematics and Mathematical Physics},
  7\penalty0 (3):\penalty0 200--217, 1967.

\bibitem[Censor and Lent(1981)]{censor1981iterative}
Yair Censor and Arnold Lent.
\newblock An iterative row-action method for interval convex programming.
\newblock \emph{Journal of Optimization Theory and Applications}, 34\penalty0
  (3):\penalty0 321--353, 1981.

\bibitem[Censor and Zenios(1992)]{censor1992proximal}
Yair Censor and Stavros~Andrea Zenios.
\newblock Proximal minimization algorithm with d-functions.
\newblock \emph{Journal of Optimization Theory and Applications}, 73\penalty0
  (3):\penalty0 451--464, 1992.

\bibitem[Cesa-Bianchi and Lugosi(2006)]{cesa2006prediction}
Nicolo Cesa-Bianchi and G{\'a}bor Lugosi.
\newblock \emph{Prediction, learning, and games}.
\newblock Cambridge University Press, 2006.

\bibitem[Defazio et~al.(2014)Defazio, Bach, and
  Lacoste-Julien]{defazio2014saga}
Aaron Defazio, Francis Bach, and Simon Lacoste-Julien.
\newblock Saga: A fast incremental gradient method with support for
  non-strongly convex composite objectives.
\newblock \emph{Advances in Neural Information Processing Systems}, 27, 2014.

\bibitem[Duchi and Namkoong(2019)]{duchi2019variance}
John Duchi and Hongseok Namkoong.
\newblock Variance-based regularization with convex objectives.
\newblock \emph{Journal of Machine Learning Research}, 20\penalty0
  (68):\penalty0 1--55, 2019.

\bibitem[Duchi and Namkoong(2021)]{duchi2021learning}
John~C Duchi and Hongseok Namkoong.
\newblock Learning models with uniform performance via distributionally robust
  optimization.
\newblock \emph{The Annals of Statistics}, 49\penalty0 (3):\penalty0
  1378--1406, 2021.

\bibitem[Duchi et~al.(2012)Duchi, Bartlett, and
  Wainwright]{duchi2012randomized}
John~C Duchi, Peter~L Bartlett, and Martin~J Wainwright.
\newblock Randomized smoothing for stochastic optimization.
\newblock \emph{SIAM Journal on Optimization}, 22\penalty0 (2):\penalty0
  674--701, 2012.

\bibitem[Duchi et~al.(2015)Duchi, Jordan, Wainwright, and
  Wibisono]{duchi2015optimal}
John~C Duchi, Michael~I Jordan, Martin~J Wainwright, and Andre Wibisono.
\newblock Optimal rates for zero-order convex optimization: The power of two
  function evaluations.
\newblock \emph{IEEE Transactions on Information Theory}, 61\penalty0
  (5):\penalty0 2788--2806, 2015.

\bibitem[Duchi et~al.(2021)Duchi, Glynn, and Namkoong]{duchi2021statistics}
John~C Duchi, Peter~W Glynn, and Hongseok Namkoong.
\newblock Statistics of robust optimization: A generalized empirical likelihood
  approach.
\newblock \emph{Mathematics of Operations Research}, 46\penalty0 (3):\penalty0
  946--969, 2021.

\bibitem[Dvinskikh et~al.(2022)Dvinskikh, Tominin, Tominin, and
  Gasnikov]{dvinskikh2022noisy}
Darina Dvinskikh, Vladislav Tominin, Iaroslav Tominin, and Alexander Gasnikov.
\newblock Noisy zeroth-order optimization for non-smooth saddle point problems.
\newblock In \emph{International Conference on Mathematical Optimization Theory
  and Operations Research}, pages 18--33. Springer, 2022.

\bibitem[Fang et~al.(2018)Fang, Li, Lin, and Zhang]{fang2018spider}
Cong Fang, Chris~Junchi Li, Zhouchen Lin, and Tong Zhang.
\newblock Spider: Near-optimal non-convex optimization via stochastic
  path-integrated differential estimator.
\newblock \emph{Advances in Neural Information Processing Systems}, 31, 2018.

\bibitem[Gao et~al.(2018)Gao, Jiang, and Zhang]{gao2018information}
Xiang Gao, Bo~Jiang, and Shuzhong Zhang.
\newblock On the information-adaptive variants of the admm: an iteration
  complexity perspective.
\newblock \emph{Journal of Scientific Computing}, 76:\penalty0 327--363, 2018.

\bibitem[Harvey et~al.(2019)Harvey, Liaw, Plan, and Randhawa]{harvey2019tight}
Nicholas~JA Harvey, Christopher Liaw, Yaniv Plan, and Sikander Randhawa.
\newblock Tight analyses for non-smooth stochastic gradient descent.
\newblock In \emph{Conference on Learning Theory}, pages 1579--1613. PMLR,
  2019.

\bibitem[Huang et~al.(2021)Huang, Wu, and Huang]{huang2021efficient}
Feihu Huang, Xidong Wu, and Heng Huang.
\newblock Efficient mirror descent ascent methods for nonsmooth minimax
  problems.
\newblock \emph{Advances in Neural Information Processing Systems},
  34:\penalty0 10431--10443, 2021.

\bibitem[Jain et~al.(2019)Jain, Nagaraj, and Netrapalli]{jain2019making}
Prateek Jain, Dheeraj Nagaraj, and Praneeth Netrapalli.
\newblock Making the last iterate of sgd information theoretically optimal.
\newblock In \emph{Conference on Learning Theory}, pages 1752--1755. PMLR,
  2019.

\bibitem[Ji et~al.(2019)Ji, Wang, Zhou, and Liang]{ji2019improved}
Kaiyi Ji, Zhe Wang, Yi~Zhou, and Yingbin Liang.
\newblock Improved zeroth-order variance reduced algorithms and analysis for
  nonconvex optimization.
\newblock In \emph{International Conference on Machine Learning}, pages
  3100--3109. PMLR, 2019.

\bibitem[Jin et~al.(2021)Jin, Zhang, Wang, and Wang]{jin2021non}
Jikai Jin, Bohang Zhang, Haiyang Wang, and Liwei Wang.
\newblock Non-convex distributionally robust optimization: Non-asymptotic
  analysis.
\newblock \emph{Advances in Neural Information Processing Systems},
  34:\penalty0 2771--2782, 2021.

\bibitem[Johnson and Zhang(2013)]{johnson2013accelerating}
Rie Johnson and Tong Zhang.
\newblock Accelerating stochastic gradient descent using predictive variance
  reduction.
\newblock \emph{Advances in Neural Information Processing Systems}, 26, 2013.

\bibitem[Kozak et~al.(2023)Kozak, Molinari, Rosasco, Tenorio, and
  Villa]{kozak2023zeroth}
David Kozak, Cesare Molinari, Lorenzo Rosasco, Luis Tenorio, and Silvia Villa.
\newblock Zeroth-order optimization with orthogonal random directions.
\newblock \emph{Mathematical Programming}, 199\penalty0 (1):\penalty0
  1179--1219, 2023.

\bibitem[Kuhn et~al.(2019)Kuhn, Esfahani, Nguyen, and
  Shafieezadeh-Abadeh]{kuhn2019wasserstein}
Daniel Kuhn, Peyman~Mohajerin Esfahani, Viet~Anh Nguyen, and Soroosh
  Shafieezadeh-Abadeh.
\newblock Wasserstein distributionally robust optimization: Theory and
  applications in machine learning.
\newblock In \emph{Operations Research \& Management Science in the Age of
  Analytics}, pages 130--166. Informs, 2019.

\bibitem[Levy et~al.(2020)Levy, Carmon, Duchi, and Sidford]{levy2020large}
Daniel Levy, Yair Carmon, John~C Duchi, and Aaron Sidford.
\newblock Large-scale methods for distributionally robust optimization.
\newblock \emph{Advances in Neural Information Processing Systems},
  33:\penalty0 8847--8860, 2020.

\bibitem[Li et~al.(2020)Li, Wang, Zhang, and Cheng]{li2020variance}
Wenjie Li, Zhanyu Wang, Yichen Zhang, and Guang Cheng.
\newblock Variance reduction on adaptive stochastic mirror descent.
\newblock \emph{arXiv preprint arXiv:2012.13760}, 2020.

\bibitem[Lin et~al.(2022)Lin, Zheng, and Jordan]{lin2022gradient}
Tianyi Lin, Zeyu Zheng, and Michael Jordan.
\newblock Gradient-free methods for deterministic and stochastic nonsmooth
  nonconvex optimization.
\newblock \emph{Advances in Neural Information Processing Systems},
  35:\penalty0 26160--26175, 2022.

\bibitem[Liu et~al.(2019)Liu, Lu, Chen, Feng, Xu, Al-Dujaili, Hong, and
  O'Reilly]{liu2019min}
Sijia Liu, Songtao Lu, Xiangyi Chen, Yao Feng, Kaidi Xu, Abdullah Al-Dujaili,
  Minyi Hong, and Una-May O'Reilly.
\newblock Min-max optimization without gradients: Convergence and applications
  to adversarial ml.
\newblock \emph{arXiv preprint arXiv:1909.13806}, 2019.

\bibitem[Liu et~al.(2020)Liu, Lu, Chen, Feng, Xu, Al-Dujaili, Hong, and
  O’Reilly]{liu2020min}
Sijia Liu, Songtao Lu, Xiangyi Chen, Yao Feng, Kaidi Xu, Abdullah Al-Dujaili,
  Mingyi Hong, and Una-May O’Reilly.
\newblock Min-max optimization without gradients: Convergence and applications
  to black-box evasion and poisoning attacks.
\newblock In \emph{International Conference on Machine Learning}, pages
  6282--6293. PMLR, 2020.

\bibitem[Luo et~al.(2020)Luo, Ye, Huang, and Zhang]{luo2020stochastic}
Luo Luo, Haishan Ye, Zhichao Huang, and Tong Zhang.
\newblock Stochastic recursive gradient descent ascent for stochastic
  nonconvex-strongly-concave minimax problems.
\newblock \emph{Advances in Neural Information Processing Systems},
  33:\penalty0 20566--20577, 2020.

\bibitem[Maheshwari et~al.(2022)Maheshwari, Chiu, Mazumdar, Sastry, and
  Ratliff]{maheshwari2022zeroth}
Chinmay Maheshwari, Chih-Yuan Chiu, Eric Mazumdar, Shankar Sastry, and Lillian
  Ratliff.
\newblock Zeroth-order methods for convex-concave min-max problems:
  Applications to decision-dependent risk minimization.
\newblock In \emph{International Conference on Artificial Intelligence and
  Statistics}, pages 6702--6734. PMLR, 2022.

\bibitem[Mertikopoulos et~al.(2018)Mertikopoulos, Lecouat, Zenati, Foo,
  Chandrasekhar, and Piliouras]{mertikopoulos2018optimistic}
Panayotis Mertikopoulos, Bruno Lecouat, Houssam Zenati, Chuan-Sheng Foo, Vijay
  Chandrasekhar, and Georgios Piliouras.
\newblock Optimistic mirror descent in saddle-point problems: Going the extra
  (gradient) mile.
\newblock \emph{arXiv preprint arXiv:1807.02629}, 2018.

\bibitem[Mohajerin~Esfahani and Kuhn(2018)]{mohajerin2018data}
Peyman Mohajerin~Esfahani and Daniel Kuhn.
\newblock Data-driven distributionally robust optimization using the
  wasserstein metric: Performance guarantees and tractable reformulations.
\newblock \emph{Mathematical Programming}, 171\penalty0 (1):\penalty0 115--166,
  2018.

\bibitem[Namkoong and Duchi(2016)]{namkoong2016stochastic}
Hongseok Namkoong and John~C Duchi.
\newblock Stochastic gradient methods for distributionally robust optimization
  with f-divergences.
\newblock \emph{Advances in Neural Information Processing Systems}, 29, 2016.

\bibitem[Nemirovski et~al.(2009)Nemirovski, Juditsky, Lan, and
  Shapiro]{nemirovski2009robust}
Arkadi Nemirovski, Anatoli Juditsky, Guanghui Lan, and Alexander Shapiro.
\newblock Robust stochastic approximation approach to stochastic programming.
\newblock \emph{SIAM Journal on Optimization}, 19\penalty0 (4):\penalty0
  1574--1609, 2009.

\bibitem[Nesterov and Spokoiny(2017)]{nesterov2017random}
Yurii Nesterov and Vladimir Spokoiny.
\newblock Random gradient-free minimization of convex functions.
\newblock \emph{Foundations of Computational Mathematics}, 17\penalty0
  (2):\penalty0 527--566, 2017.

\bibitem[Nguyen et~al.(2017{\natexlab{a}})Nguyen, Liu, Scheinberg, and
  Tak{\'a}{\v{c}}]{nguyen2017sarah}
Lam~M Nguyen, Jie Liu, Katya Scheinberg, and Martin Tak{\'a}{\v{c}}.
\newblock Sarah: A novel method for machine learning problems using stochastic
  recursive gradient.
\newblock In \emph{International Conference on Machine Learning}, pages
  2613--2621. PMLR, 2017{\natexlab{a}}.

\bibitem[Nguyen et~al.(2017{\natexlab{b}})Nguyen, Liu, Scheinberg, and
  Tak{\'a}{\v{c}}]{nguyen2017stochastic}
Lam~M Nguyen, Jie Liu, Katya Scheinberg, and Martin Tak{\'a}{\v{c}}.
\newblock Stochastic recursive gradient algorithm for nonconvex optimization.
\newblock \emph{arXiv preprint arXiv:1705.07261}, 2017{\natexlab{b}}.

\bibitem[Nguyen et~al.(2021)Nguyen, Scheinberg, and
  Tak{\'a}{\v{c}}]{nguyen2021inexact}
Lam~M Nguyen, Katya Scheinberg, and Martin Tak{\'a}{\v{c}}.
\newblock Inexact sarah algorithm for stochastic optimization.
\newblock \emph{Optimization Methods and Software}, 36\penalty0 (1):\penalty0
  237--258, 2021.

\bibitem[Nguyen et~al.(2019)Nguyen, Shafieezadeh~Abadeh, Yue, Kuhn, and
  Wiesemann]{nguyen2019optimistic}
Viet~Anh Nguyen, Soroosh Shafieezadeh~Abadeh, Man-Chung Yue, Daniel Kuhn, and
  Wolfram Wiesemann.
\newblock Optimistic distributionally robust optimization for nonparametric
  likelihood approximation.
\newblock \emph{Advances in Neural Information Processing Systems}, 32, 2019.

\bibitem[Paul et~al.(2023)Paul, Mahindrakar, and Kalaimani]{paul2023robust}
Anik~Kumar Paul, Arun~D Mahindrakar, and Rachel~K Kalaimani.
\newblock Robust analysis of almost sure convergence of zeroth-order mirror
  descent algorithm.
\newblock \emph{IEEE Control Systems Letters}, 2023.

\bibitem[Qi et~al.(2021)Qi, Guo, Xu, Jin, and Yang]{qi2021online}
Qi~Qi, Zhishuai Guo, Yi~Xu, Rong Jin, and Tianbao Yang.
\newblock An online method for a class of distributionally robust optimization
  with non-convex objectives.
\newblock \emph{Advances in Neural Information Processing Systems},
  34:\penalty0 10067--10080, 2021.

\bibitem[Rafique et~al.(2022)Rafique, Liu, Lin, and Yang]{rafique2022weakly}
Hassan Rafique, Mingrui Liu, Qihang Lin, and Tianbao Yang.
\newblock Weakly-convex--concave min--max optimization: provable algorithms and
  applications in machine learning.
\newblock \emph{Optimization Methods and Software}, 37\penalty0 (3):\penalty0
  1087--1121, 2022.

\bibitem[Rando et~al.(2022)Rando, Molinari, Villa, and
  Rosasco]{rando2022stochastic}
Marco Rando, Cesare Molinari, Silvia Villa, and Lorenzo Rosasco.
\newblock Stochastic zeroth order descent with structured directions.
\newblock \emph{arXiv preprint arXiv:2206.05124}, 2022.

\bibitem[Reddi et~al.(2016)Reddi, Hefny, Sra, Poczos, and
  Smola]{reddi2016stochastic}
Sashank~J Reddi, Ahmed Hefny, Suvrit Sra, Barnabas Poczos, and Alex Smola.
\newblock Stochastic variance reduction for nonconvex optimization.
\newblock In \emph{International Conference on Machine Learning}, pages
  314--323. PMLR, 2016.

\bibitem[Sadiev et~al.(2021)Sadiev, Beznosikov, Dvurechensky, and
  Gasnikov]{sadiev2021zeroth}
Abdurakhmon Sadiev, Aleksandr Beznosikov, Pavel Dvurechensky, and Alexander
  Gasnikov.
\newblock Zeroth-order algorithms for smooth saddle-point problems.
\newblock In \emph{International Conference on Mathematical Optimization Theory
  and Operations Research}, pages 71--85. Springer, 2021.

\bibitem[Sagawa et~al.(2019)Sagawa, Koh, Hashimoto, and
  Liang]{sagawa2019distributionally}
Shiori Sagawa, Pang~Wei Koh, Tatsunori~B Hashimoto, and Percy Liang.
\newblock Distributionally robust neural networks for group shifts: On the
  importance of regularization for worst-case generalization.
\newblock \emph{arXiv preprint arXiv:1911.08731}, 2019.

\bibitem[Shamir and Zhang(2013)]{shamir2013stochastic}
Ohad Shamir and Tong Zhang.
\newblock Stochastic gradient descent for non-smooth optimization: Convergence
  results and optimal averaging schemes.
\newblock In \emph{International Conference on Machine Learning}, pages 71--79.
  PMLR, 2013.

\bibitem[Shapiro(2017)]{shapiro2017distributionally}
Alexander Shapiro.
\newblock Distributionally robust stochastic programming.
\newblock \emph{SIAM Journal on Optimization}, 27\penalty0 (4):\penalty0
  2258--2275, 2017.

\bibitem[Shen et~al.(2023)Shen, Wang, and Xu]{shen2023zeroth}
Jingjing Shen, Ziqi Wang, and Zi~Xu.
\newblock Zeroth-order single-loop algorithms for nonconvex-linear minimax
  problems.
\newblock \emph{Journal of Global Optimization}, 87\penalty0 (2):\penalty0
  551--580, 2023.

\bibitem[Sinha et~al.(2017)Sinha, Namkoong, Volpi, and
  Duchi]{sinha2017certifying}
Aman Sinha, Hongseok Namkoong, Riccardo Volpi, and John Duchi.
\newblock Certifying some distributional robustness with principled adversarial
  training.
\newblock \emph{arXiv preprint arXiv:1710.10571}, 2017.

\bibitem[Sugiyama et~al.(2007)Sugiyama, Krauledat, and
  M{\"u}ller]{sugiyama2007covariate}
Masashi Sugiyama, Matthias Krauledat, and Klaus-Robert M{\"u}ller.
\newblock Covariate shift adaptation by importance weighted cross validation.
\newblock \emph{Journal of Machine Learning Research}, 8\penalty0 (5), 2007.

\bibitem[Wai et~al.(2019)Wai, Hong, Yang, Wang, and Tang]{wai2019variance}
Hoi-To Wai, Mingyi Hong, Zhuoran Yang, Zhaoran Wang, and Kexin Tang.
\newblock Variance reduced policy evaluation with smooth function
  approximation.
\newblock \emph{Advances in Neural Information Processing Systems}, 32, 2019.

\bibitem[Wang et~al.(2018)Wang, Ji, Zhou, Liang, and
  Tarokh]{wang2018spiderboost}
Zhe Wang, Kaiyi Ji, Yi~Zhou, Yingbin Liang, and Vahid Tarokh.
\newblock Spiderboost: A class of faster variance-reduced algorithms for
  nonconvex optimization.
\newblock \emph{arXiv}, 2018, 2018.

\bibitem[Wang et~al.(2023)Wang, Balasubramanian, Ma, and
  Razaviyayn]{wang2023zeroth}
Zhongruo Wang, Krishnakumar Balasubramanian, Shiqian Ma, and Meisam Razaviyayn.
\newblock Zeroth-order algorithms for nonconvex--strongly-concave minimax
  problems with improved complexities.
\newblock \emph{Journal of Global Optimization}, 87\penalty0 (2):\penalty0
  709--740, 2023.

\bibitem[Xu et~al.(2023)Xu, Wang, Wang, and Dai]{xu2023zeroth}
Zi~Xu, Ziqi Wang, Junlin Wang, and Yuhong Dai.
\newblock Zeroth-order alternating gradient descent ascent algorithms for a
  class of nonconvex-nonconcave minimax problems.
\newblock \emph{Journal of Machine Learning Research}, 24\penalty0
  (313):\penalty0 1--25, 2023.

\bibitem[Xu et~al.(2024)Xu, Wang, Shen, and Dai]{xu2021derivative}
Zi~Xu, Ziqi Wang, Jingjing Shen, and Yuhong Dai.
\newblock Derivative-free alternating projection algorithms for general
  nonconvex-concave minimax problems.
\newblock \emph{SIAM Journal on Optimization}, 34\penalty0 (2):\penalty0
  1879--1908, 2024.

\bibitem[Zhang and Tu(2023)]{zhang2023efficient}
Lijun Zhang and Weiwei Tu.
\newblock Efficient stochastic approximation of minimax excess risk
  optimization.
\newblock \emph{arXiv preprint arXiv:2306.00026}, 2023.

\bibitem[Zinkevich(2003)]{zinkevich2003online}
Martin Zinkevich.
\newblock Online convex programming and generalized infinitesimal gradient
  ascent.
\newblock In \emph{International Conference on Machine Learning}, pages
  928--936, 2003.

\end{thebibliography}
	
\end{document}